\documentclass[12pt]{amsart}
\usepackage{amsmath,amssymb,amsthm,color, euscript, enumerate,comment}
\usepackage{dsfont}
\usepackage{longtable}

\newcommand{\1}{\mathds{1}}

\newcommand{\maru}[1]{{\ooalign{\hfil#1\/\hfil\crcr
\raise.167ex\hbox{\mathhexbox20D}}}}

\newcommand{\ruby}[2]{%
 \leavevmode
 \setbox0=\hbox{#1}%
 \setbox1=\hbox{\tiny #2}%
 \ifdim\wd0>\wd1 \dimen0=\wd0 \end{lemma}se \dimen0=\wd1 \fi
 \hbox{%
   \kanjiskip=0pt plus 2fil
   \xkanjiskip=0pt plus 2fil
   \vbox{%
     \hbox to \dimen0{%
       \tiny \hfil#2\hfil}%
     \nointerlineskip
     \hbox to \dimen0{\mathstrut\hfil#1\hfil}}}}

\newcommand{\Z}{\mathbb{Z}}
\newcommand{\C}{\mathbb{C}}
\newcommand{\R}{\mathbb{R}}

\newcommand{\Q}{\mathbb{Q}}

\newcommand{\g}{\mathfrak{g}}

\newcommand{\GO}{\mathrm{GO}}

\newcommand{\PO}{\mathrm{P}\Omega}

\newcommand{\Dih}{\mathrm{Dih}}
\newcommand{\rank}{\mathrm{rank}\,}

\newcommand{\h}{\mathfrak{h}}
\newcommand{\F}{\mathbb{F}}

\newcommand{\Sym}{{\mathfrak{S}}}
\newcommand{\Alt}{{\mathfrak{A}}}
\newcommand{\Imm}{\mathrm{Im}\,}
\newcommand{\Ker}{\mathrm{Ker}\,}

\newcommand{\aut}{\mathrm{Aut}\,}
\newcommand{\Aut}{\mathrm{Aut}\,}
\newcommand{\Inn}{\mathrm{Inn}\,}

\newcommand{\AGL}{\mathrm{AGL}}
\newcommand{\Com}{\mathrm{Com}}

\newcommand{\Stab}{\mathrm{Stab}}

\newcommand{\irr}{\mathrm{Irr}\,}

\makeatletter \@addtoreset{equation}{section}

\theoremstyle{plain}
\newtheorem{theorem}{Theorem}[section]
\newtheorem{proposition}[theorem]{Proposition}
\newtheorem{lemma}[theorem]{Lemma}

\theoremstyle{definition}
\newtheorem{definition}[theorem]{Definition}

\theoremstyle{remark}
\newtheorem{remark}[theorem]{Remark}
\numberwithin{equation}{section}

\title[Automorphism groups and uniqueness of holomorphic VOAs]{Automorphism groups and uniqueness of holomorphic vertex operator algebras of central charge $24$}

\author{Koichi Betsumiya}
  \address[K. Betsumiya] {Graduate School of Science and Technology,
Hirosaki University, Hirosaki 036-8561, Japan} 
  \email{betsumi@hirosaki-u.ac.jp}
\author{Ching Hung Lam} %
  \address[C. H. Lam] {Institute of Mathematics, Academia Sinica, Taipei 10617, Taiwan} 
  \email{chlam@math.sinica.edu.tw}
\author[H. Shimakura]{Hiroki Shimakura}%
\address[H. Shimakura]{Graduate School of Information Sciences, Tohoku University, Sendai, 980-8579, Japan }%
\email {shimakura@tohoku.ac.jp}%
\date{}
\thanks{C.H. Lam was partially supported by a research grant AS-IA-107-M02 of Academia Sinica  and MoST grants  107-2115-M-001-003-MY3 and 110-2115-M-001-011-MY3 of Taiwan}
\thanks{H. Shimakura was partially supported by JSPS KAKENHI Grant Numbers JP19KK0065 and JP20K03505.}

\newcommand{\sfr}[2]{\leavevmode\kern-.1em
  \raise.5ex\hbox{\the\scriptfont0 #1}\kern-.1em
  /\kern-.15em\lower.25ex\hbox{\the\scriptfont0 #2}}

\newcommand{\Out}{{\mathrm{Out}}\,}
\newcommand{\OOut}{{\mathrm{Out}}\,}

\begin{document}

\begin{abstract}
We describe the automorphism groups of all holomorphic vertex operator algebras of central charge $24$ with non-trivial weight one Lie algebras by using their constructions as simple current extensions. 
We also confirm a conjecture of G. H\"ohn on the numbers of holomorphic vertex operator algebras of central charge $24$ obtained as inequivalent simple current extensions of certain vertex operator algebras, which gives another proof of the uniqueness of holomorphic vertex operator algebras of central charge $24$ with non-trivial weight one Lie algebras.
\end{abstract}

\maketitle


\section{Introduction}
The classification of (strongly regular) holomorphic vertex operator algebras (VOAs) of central charge $24$ has been completed except for the characterization of the moonshine VOA; more precisely, the following are proved:
\begin{enumerate}[(a)]
\item The weight one Lie algebra of a holomorphic VOA of central charge $24$ is $0$, $24$-dimensional abelian or one of the $69$ semisimple Lie algebras in \cite{Sc93}; the list of these $71$ Lie algebras is called \emph{Schellekens' list}.
\item For any Lie algebra $\g$ in Schellekens' list, there exists a holomorphic VOA of central charge $24$ whose weight one Lie algebra is isomorphic to $\g$.
\item The isomorphism class of a holomorphic VOA $V$ of central charge $24$ with $V_1\neq0$ is uniquely determined by the weight one Lie algebra structure on $V_1$.
\end{enumerate}
Item (a) was proved in \cite{Sc93,EMS} (see \cite{ELMS} for another proof).
Items (b) and (c) were proved by using case by case analysis (see \cite{LS19} and \cite[Introduction]{LS7}); several uniform approaches for (b) and (c) are also discussed in \cite{Ho2,MS1, MS2, HM,CLM,LM}.

For our purpose, we first recall a uniform approach proposed by \cite{Ho2} briefly.
Let $V$ be a holomorphic VOA of central charge $24$.
Then $V_1$ is a reductive Lie algebra of rank at most $24$ (\cite{DM04b});
if $\rank V_1=24$, then $V$ is isomorphic to a Niemeier lattice VOA (\cite{DM04b}) and that if $\rank V_1=0$, equivalently, $V_1=0$, then it is conjectured in \cite{FLM} that $V$ is  isomorphic to the moonshine VOA.
In this article, we assume that $0<\rank V_1<24$.
Then $V_1$ is semisimple and the subVOA $\langle V_1\rangle$ generated by $V_1$ has central charge $24$, that means $\langle V_1\rangle$ is a full subVOA of $V$ (\cite{DM04a}).
Let $\h$ be a Cartan subalgebra of $V_1$.
The following item (d) was essentially proved in \cite{Ho2}; the necessary assumptions are confirmed in \cite{Lam20} 
(see also \cite{ELMS} for another proof):

\begin{enumerate}[(a)]
\setcounter{enumi}{3}
\item 
The commutant $W=\Com_V(\h)$ of $\h$ in $V$ is isomorphic to the fixed-point subVOA $V_{\Lambda_g}^{\hat{g}}$ of the lattice VOA $V_{\Lambda_g}$ with respect to a (standard) lift $\hat{g}\in\Aut(V_{\Lambda_g})$ of an isometry $g_{|\Lambda_g}$ of $\Lambda_g$, where $g$ is an isometry of the Leech lattice $\Lambda$ in  
one of the following $10$ conjugacy classes (as the notations in \cite{ATLAS})
$$2A,2C,3B,4C,5B,6F,6G,7B,8E\ \text{and}\ 10F,$$ 
and $\Lambda_g$ is the coinvariant lattice of $g$ (see Definition \ref{Def:coinv}).
In addition, the conjugacy class of $g$ is uniquely determined by the Lie algebra structure of $V_1$.
\end{enumerate}
The commutant $\Com_V(W)$ of $W$ in $V$ is isomorphic to a lattice VOA $V_{L}$.
In fact, the lattice $L=L_\g$, called the orbit lattice in \cite{Ho2}, is also uniquely determined by the Lie algebra structure of $\g=V_1$.
Note that some non-isomorphic Lie algebras in Schellekens' list give isometric orbit lattices.
Since both $V_L$ and $W$ have group-like fusion (\cite{Do,Lam20}), $V$ is a simple current extension of $V_{L}\otimes W$.
In order to prove (b), it suffices to construct all holomorphic VOAs of central charge $24$ as simple current extensions of $V_L\otimes W$, which was discussed in \cite[Theorem 4.4]{Ho2} under some assumptions (see also \cite{Lam20}).
In order to prove (c), it suffices to classify all holomorphic VOAs of central charge $24$ as simple current extensions of $V_L\otimes W$ up to isomorphism, which
can be proved by confirming the conjecture \cite[Conjecture 4.8]{Ho2} on the number of inequivalent simple current extensions of $V_L\otimes W$ that form holomorphic VOAs of central charge $24$.

Another question is to determine the automorphism groups of holomorphic VOAs of central charge $24$.
Our strategy is to describe the automorphism group of a VOA via its weight one Lie algebra.
Let $T$ be a VOA of CFT-type. 
Set
 $$K(T):=\{g\in\Aut(T)\mid g=id\ \text{on}\ T_1\},$$
the subgroup of $\Aut(T)$ which acts trivially on $T_1$.
Let ${\Aut}(T)_{|T_1}$ denote the restriction of $\Aut(T)$ to $T_1$.
Then $${\Aut}(T)_{|T_1}\cong\Aut(T)/K(T)\subset\Aut(T_1).$$
Recall that $\Aut(T)$ contains the inner automorphism group $\Inn(T)$, the normal subgroup generated by inner automorphisms $\{\exp(a_{(0)})\mid a\in T_1\}$.
Let ${\Inn}(T)_{|T_1}$ denote the restriction of $\Inn(T)$ to $T_1$, that is,  $${\Inn}(T)_{|T_1}\cong \Inn(T)/ (K(T)\cap\Inn(T)).$$
Clearly, ${\Inn}(T)_{|T_1}$ is isomorphic to the inner automorphism group $\Inn(T_1)$ of $T_1$.

Define $$\Out(T):=\Aut(T)/\Inn(T).$$ In Proposition \ref{L:idh}, we will show that $$K(T)\subset\Inn(T)$$  when $T$ is a holomorphic VOA of central charge $24$ with $T_1\neq0$. 
For these cases, $$\Out(T)=\Aut(T)/\Inn(T)\cong {\Aut}(T)_{|T_1}/{\Inn}(T)_{|T_1},$$ and $\Out(T)$ is a subgroup of the outer automorphism group $\OOut(T_1)\cong\Aut(T_1)/\Inn(T_1)$ of the weight one Lie algebra $T_1$.
Note that the inclusion $K(T)\subset\Inn(T)$ does not hold in general; for example, the moonshine VOA $V^\natural$ satisfies $K(V^\natural)=\Aut(V^\natural)\neq1$ and $\Inn(V^\natural)=1$.
Therefore, for a holomorphic VOA $T$ of central charge $24$ with $T_1\neq0$, we have 
$$\Aut(T)\cong K(T).({\Inn}(T)_{|T_1}.\Out(T)),$$
where $A.B$ means a group $G$ that contains a normal subgroup $A$ with $G/A\cong B$.
Hence $\Aut(T)$ is roughly described by the groups ${\Inn}(T)_{|T_1}$, $K(T)$ and $\Out(T)$.
Note that ${\Inn}(T)_{|T_1}(\cong\Inn(T_1))$ is well-studied. 

For a holomorphic lattice VOA $V_N$
associated with a Niemeier lattice $N$, the groups $K(V_N)$ and $\Out(V_N)$ can be easily determined by the description of $\Aut(V_N)$ in \cite{DN99} (see Remark \ref{R:KOknown}).
For the $14$ holomorphic VOAs $V_N^{{\rm orb}(\theta)}$
 obtained by applying the $\Z_2$-orbifold construction to $V_N$ and a lift $\theta$ of the $-1$-isometry of $N$, the groups $K(V_N^{{\rm orb}(\theta)})$ and $\Out(V_N^{{\rm orb}(\theta)})$ are determined in \cite{Sh20} by using the explicit constructions.
For some 
holomorphic VOAs $V$ of central charge $24$,
 $K(V)$ is determined in \cite{LS7} by using the module structure of $V$ over the simple affine VOA generated by $V_1$ and its fusion product.

In this article, we assume (a), (b) and (d) and describe the orbit lattice $L$ for each semisimple Lie algebra $\g$ in Schellekens' list of rank less than $24$; this shows that the orbit lattice $L=L_\g$ is uniquely determined by $\g$, up to isometry.
Note that the orbit lattices have been described in \cite{Ho2} by using Niemeier lattices.
In addition, we determine the groups $K(V)$ and $\Out(V)$ for all holomorphic VOAs $V$ of central charge $24$ with $0<\rank V_1<24$ by using $O(L)$, $\Aut(W)$ and the constructions of $V$ as simple current extensions of $V_L\otimes W$.
Note that the automorphism groups $\Aut(W)$ for all $10$ VOAs $W$ in (d) have been determined in \cite{GO+102,Sh04,CLS,Lam18b,LamCFT,BLS}.
In particular, we prove the following (see Remark \ref{R:KOknown}, Proposition \ref{L:idh} and Tables \ref{T:2to6}, \ref{table3to8}, \ref{table5to6}, \ref{table7to5}, \ref{2c1}, \ref{2c1b}, \ref{OK4C}, \ref{OK6E}, \ref{OK8E}, \ref{OK6G} and \ref{tablelevel10}):

\begin{theorem}\label{T:main11} Let $V$ be a strongly regular holomorphic VOA of central charge $24$ with $V_1\neq0$.
Then $K(V)\subset \Inn(V)$.
Moreover, the group structures of $K(V)$ and $\Out(V)$ are given as in Table \ref{T:KOUT}.
Here the genus symbol $A,B,\dots,K$ in the table are used in \cite{Ho2}
\end{theorem}
\begin{footnotesize}
\begin{longtable}[c]{|c|c|c|c|c||c|c|c|c|c|}
\caption{$K(V)$ and $\Out(V)$ for holomorphic VOAs $V$ of central charge $24$ with $V_1\neq0$} \label{T:KOUT}\\
\hline 
No.&Genus&$V_1$&$\Out(V)$& $K(V)$ &No.&Genus&$V_1$&$\Out(V)$& $K(V)$ \\ \hline 
\hline 

$1$&$A$&$U(1)^{24}$&$Co_0$&$\C^{24}$&
$48$&$B$&$C_{6,1}^2B_{4,1}$ &$\Z_2$&$\Z_2$\\

$15$&&$A_{1,1}^{24}$&$M_{24}$&$\Z_2^{12}$&
$50$&&$D_{9,2}A_{7,1}$ &$\Z_2$&$\Z_8$\\

$24$&&$A_{2,1}^{12}$&$\Z_2.M_{12}$&$\Z_3^{6}$&
$52$&&$C_{8,1}F_{4,1}^2$ &$\Z_2$&$1$\\

$30$&&$A_{3,1}^{8}$&$\Z_2.AGL_3(2)$&$\Z_4^{4}$&
$53$&&$E_{7,2}B_{5,1}F_{4,1}$ &$1$&$\Z_2$\\

$37$&&$A_{4,1}^{6}$&$\Z_2.\Sym_5$&$\Z_5^{3}$&
$56$&&$C_{10,1}B_{6,1}$ &$1$&$\Z_2$\\

$42$&&$D_{4,1}^{6}$&$\Z_3.\Sym_6$&$\Z_2^{6}$&
$62$&&$B_{8,1}E_{8,2}$&$1$&$\Z_2$\\
\cline{6-10}

$43$&&$A_{5,1}^4D_{4,1}$&$\Z_2.\Sym_4$&$\Z_2^3\times\Z_3^{2}$&
$6$&$C$&$A_{2,3}^{6}$  & $\Sym_6$&  $\Z_3$\\

$46$&&$A_{6,1}^{4}$&$\Z_2.\Alt_4$&$\Z_7^{2}$&
$17$&&${A_{5,3}}{D_{4,3}}A_{1,1}^{3}$ & $\Sym_3$& $\Z_2^3$\\

$49$&&$A_{7,1}^{2}D_{5,1}^2$&$\Z_2.\Z_2^2$&$\Z_4\times\Z_8$&
$27$&&$A_{8,3}A_{2,1}^2$    & $\Z_2$&$\Z_3^2$\\

$51$&&$A_{8,1}^{3}$&$\Z_2.\Sym_3$&$\Z_9\times\Z_3$&
$32$&&$ E_{6,3}{G_{2,1}}^3$&  $\Sym_3$& $1$ \\

$54$&&$D_{6,1}^{4}$&$\Sym_4$&$\Z_2^{4}$&
$34$&&$  {D_{7,3}}{A_{3,1}}{G_{2,1}}$&  $1$& $\Z_4$\\

$55$&&$A_{9,1}^2D_{6,1}$&$\Z_2.\Z_2$&$\Z_2^{2}\times\Z_5$&
$45$&&$E_{7,3}A_{5,1}$ & $1$ & $\Z_6$ \\
\cline{6-10}

$58$&&$E_{6,1}^{4}$&$\Z_2.\Sym_4$&$\Z_3^{2}$&
$2$&$D$ &$A_{1,4}^{12}$ & $M_{12}$&   $\Z_2$
\\ 

$59$&&$A_{11,1}D_{7,1}E_{6,1}$&$\Z_2$&$\Z_3\times \Z_4$&
$12$ &&$B_{2,2}^6$ & $\Sym_5$& $\Z_2$\\

$60$&&$A_{12,1}^{2}$&$\Z_2.\Z_2$&$\Z_{13}$&
$23$& &$B^4_{3,2}$ &$\Alt_4$&   $\Z_2$
\\

$61$&&$D_{8,1}^{3}$&$\Sym_3$&$\Z_2^{3}$&
$29$& &$B_{4,2}^3$ &  $\Sym_3$& $\Z_2$\\

$63$&&$A_{15,1}D_{9,1}$&$\Z_2$&$\Z_8$&
$41$ &&$B^2_{6,2}$ & $\Z_2$& $\Z_2$\\ 

$64$&&$D_{10,1}E_{7,1}^2$&$\Z_2$&$\Z_2^{2}$&
$57$ &&$B_{12,2}$ & $1$ &  $\Z_2$ 
\\ 

$65$&&$A_{17,1}E_{7,1}$&$\Z_2$&$\Z_2\times\Z_3$&
$13$& &$D_{4,4}A^4_{2,2}$ & $2.\Sym_4$&  $\Z_3^2$\\ 

$66$&&$D_{12,1}^{2}$&$\Z_2$&$\Z_2^2$&
$22$ &&$C_{4,2}A^2_{4,2}$& $\Z_4$& $\Z_5$\\

$67$&&$A_{24}$&$\Z_2$&$\Z_5$&
$36$ &&$ A_{8,2}F_{4,2}$ & $\Z_2$  &   $\Z_3$ \\ 
\cline{6-10}

$68$&&$E_{8,1}^{3}$&$\Sym_3$&$1$&
$7$&$E$ &$A_{3,4}^3A_{1,2}$  &$\Z_2^2{:}\Sym_3$& $\Z_2$ \\ 

$69$&&$D_{16,1}E_{8,1}$&$1$&$\Z_2$&
$18$ &&$A_{7,4}A_{1,1}^3$  &$\Z_2$  & $\Z_2^3$\\

$70$&&$D_{24,1}$&$1$&$\Z_2$&
$19$ &&$D_{5,4}C_{3,2}A_{1,1}^2$  &$\Z_2$& $\Z_2^3$\\ 
\cline{1-5}

$5$&$B$&$A_{1,2}^{16}$&$\AGL_4(2)$&  $\Z_2^5$& 
$28$ &&$E_{6,4}A_{2,1}B_{2,1}$ &  $1$& $\Z_6$\\

$16$&&$A_{3,2}^4A_{1,1}^4$& $W(D_4)$& $\Z_2^3\times\Z_4$ &
$35$ &&$C_{7,2}A_{3,1}$  & $1$ & $\Z_2^2$\\ \cline{6-10} 

$25$&&$D_{4,2}^2C_{2,1}^4$&$\Z_2\times\Sym_4$&$\Z_2^3$ &
$9$&$F$&$ A_{4,5}^2$ & $\Z_2^2$ & $1$ \\

$26$&&$A_{5,2}^2C_{2,1}A_{2,1}^2$ &$\Dih_8$  &$\Z_3\times\Z_6$& 
$20$&&$D_{6,5}A_{1,1}^2$ & $1$& $\Z_2^2$\\  \cline{6-10}

$31$&&$D_{5,2}^2A_{3,1}^2$ &$\Dih_8$&$\Z_4^2$&
$8$ &$G$& $A_{5,6}B_{2,3}A_{1,2}$  
&$\Z_2$& $\Z_2$\\  

$33$&&$A_{7,2}C_{3,1}^2A_{3,1}$&$\Z_2^2$ &$\Z_2\times \Z_4$&
$21$& & $C_{5,3}G_{2,2}A_{1,1}$ & $1$ & $\Z_2$
\\  \cline{6-10}
 
$38$&&$C_{4,1}^4$ &$\Sym_4$&$\Z_2$ &
$11$&$H$&$ A_{6,7}$ & $1$ & $1$\\  \cline{6-10}

$39$&&$D_{6,2}C_{4,1}B_{3,1}^2$& $\Z_2$&$\Z_2^2$&
$10$&$I$ & $D_{5,8}A_{1,2}$   & $1$ &  $\Z_2$
\\   \cline{6-10}

$40$&&$A_{9,2}A_{4,1}B_{3,1}$&$\Z_2$ &$\Z_{10}$&
$3$&$J$ &$D_{4,12}A_{2,6}$   & $\Sym_3$ & $1$\\

$44$&&$E_{6,2}C_{5,1}A_{5,1}$
&$\Z_2$&$\Z_6$&
$14$& &$F_{4,6}A_{2,2}$   & $1$ & $\Z_3$
\\\cline{6-10}

$47$&&$D_{8,2}B_{4,1}^2$&$\Z_2$&$\Z_2^2$ &
$4$&$K$ &$C_{4,10}$  & $1$ &$1$\\ \hline
\end{longtable}
\end{footnotesize}

\begin{remark}
The former assertion $K(V)\subset\Inn(V)$ of Theorem \ref{T:main11} is proved in Proposition \ref{L:idh} by using the fact that for $W$ in (d), $\Aut(W)$ acts faithfully on the set of isomorphism classes of irreducible $W$-modules (see Theorem \ref{T:inj}).
\end{remark}

\begin{remark}
We do \emph{not} describe the embedding of $\Out(V)$ into $\OOut(V_1)$.
However, in Appendix A, we describe the subgroup $\Out_1(V)$ of $\Out(V)$ which preserves every simple ideal of $V_1$, and the quotient $\Out_2(V):=\Out(V)/\Out_1(V)$, which is a permutation group on the set of simple ideals of $V_1$.

\end{remark}

By using $O(L)$ and $\Aut(W)$, we also prove the following (see Propositions \ref{P:uni2-8}, \ref{P:uni5B}, \ref{P:uni2Ca}, \ref{P:uni2Cb}, \ref{P:uni6G} and \ref{P:uni10F}).
In particular, we confirm \cite[Conjecture 4.8]{Ho2}.

\begin{theorem}\label{T:main13}
Let $g\in O(\Lambda)$ in one of the $10$ conjugacy classes in (d).
Set $W=V_{\Lambda_g}^{\hat{g}}$.
Let $L$ be an even lattice such that 
there exists a simple current extension of $V_L\otimes W$ which forms a holomorphic VOA $V$ of central charge $24$; in addition, $V_L=\Com_V(W)$ and $W=\Com_V(V_L)$.
Then, we have the following results.
\begin{enumerate}[{\rm (1)}]
\item  Assume  that the conjugacy class of $g$ is $2A,3B,4C,5B,6F,7B,8E$ or $10F$.
Then, there exists a unique holomorphic VOA of central charge $24$ obtained as an inequivalent simple current extension of $V_{L}\otimes W$, up to isomorphism, for each possible $L$.

\item Assume that the conjugacy class of $g$ is $2C$.
Then $L\cong \sqrt2D_{12}$ or $\sqrt2E_8\sqrt2D_4$.
In addition, there exist exactly $6$ (resp. $3$) semisimple Lie algebras in Schellekens' list such that the associated orbit lattices are isometric to $\sqrt2D_{12}$ (resp. $\sqrt2E_8\sqrt2D_4$), and there exist exactly $6$ (resp. $3$) holomorphic VOAs of central charge $24$ obtained as inequivalent simple current extensions of $V_{\sqrt2D_{12}}\otimes W$ (resp. $V_{\sqrt2E_8\sqrt2D_4}\otimes W$), up to isomorphism.
\item Assume that the conjugacy class of $g$ is $6G$.
Then $L\cong \sqrt{6}D_4\sqrt2A_2$, and there exist exactly $2$ semisimple Lie algebras in Schellekens' list such that the associated orbit lattices are isometric to $\sqrt{6}D_4\sqrt2A_2$.
In addition, there exist exactly $2$ holomorphic VOAs of central charge $24$ obtained as inequivalent simple current extensions of $V_{\sqrt{6}D_4\sqrt2A_2}\otimes W$, up to isomorphism.
\end{enumerate}
\end{theorem}

\begin{remark} The assumption on $L$ and $W$ in Theorem \ref{T:main13} 
is equivalent to the conditions that $(\mathcal{D}(L),q_L)\cong (\irr(W),-q_W)$ as quadratic spaces and the sum of the rank of $L$ and the central charge of $W$ is $24$ (see Section \ref{S:31}).
\end{remark}

It follows that a semisimple Lie algebra $\g$ in Schellekens' list of rank less than $24$ determines a unique equivalence class of 
a simple current extension of $V_L\otimes W$ which forms 
a holomorphic VOA $V$ of central charge $24$ with $V_1\cong\g$.
Hence Theorem \ref{T:main13}
and the characterization of Niemeier lattice VOAs in \cite{DM04b} give another proof of (c) (see \cite[Section 4.3]{Ho2}).

\medskip

Let us explain the main ideas for determining the groups $K(V)$ and $\Out(V)$ for holomorphic VOAs $V$ of central charge $24$ with $0<\rank V_1<24$.
As we mentioned above, $V$ is a simple current extension of $V_{L}\otimes W$.
Recall from \cite{Do} (resp. \cite{Lam20}) that all irreducible $V_L$-modules (resp. irreducible $W$-modules) are simple current modules.
Hence the set $\irr(V_L)$ (resp. $\irr(W)$) of their isomorphism classes has group-like fusion, that is, it forms an abelian group under the fusion product.
In addition,  the map $q_{V_L}$ (resp. $q_W$) from $\irr(V_L)$ (resp. $\irr(W)$) to $\Q/\Z$ defined by conformal weights modulo $\Z$ is a quadratic form (\cite{EMS}).
It is well-known that $(\irr(V_L),q_{V_L})$ is isometric to the quadratic space $(\mathcal{D}(L),q_L)$ on the discriminant group $\mathcal{D}(L)=L^*/L$ with the quadratic form $q_L(v+L)=\langle v|v\rangle/2+\Z$.
Since $V$ is holomorphic, there exists a bijection $\varphi$ from $\mathcal{D}(L)$ to $\irr(W)$ such that 
for any ${\lambda+L}\in\mathcal{D}(L)$, $V_{\lambda+L}\otimes\varphi(\lambda+L)$ appears as a $V_L\otimes W$-submodule of $V$.
Note that $\varphi$ is an isometry from $(\mathcal{D}(L),q_L)$ to $(\irr(W),-q_W)$.
Then $S_\varphi=\{(V_{\lambda+L},\varphi(\lambda+L))\mid \lambda+L\in\mathcal{D}(L)\}$ is a maximal totally isotropic subspace of $(\irr(V_L),q_{V_L})\oplus(\irr(W),q_W)$.

Since the group $K(V)$ acts trivially on the (fixed) Cartan subalgebra $\h$, it preserves $V_L\otimes W$ and $S_\varphi$.
In addition, the restriction of $K(V)$ to $V_L$ is a subgroup of the inner automorphisms associated with $\h$ and preserves every element in $\irr(V_L)$.  
Hence the restriction of $K(V)$ to $W$ also   preserves every element in $\irr(W)$ via the isometry $\varphi$; since the action of $\Aut(W)$ on $\irr(W)$ is faithful, the restriction of $K(V)$ to $W$ must be the identity.
Note that the subgroup which acts trivially on $V_L\otimes W$ is the dual $S_\varphi^*$ of $S_\varphi$, which is contained in $\Inn(V)$.
Hence $K(V)$ is contained in $\Inn(V)$.
In addition, we describe $K(V)$ in terms of $L$ and the root lattice of $V_1$ (Proposition \ref{L:KV2}).

By the transitivity of $\Inn(V)$ on Cartan subalgebras of $V_1$, $\Out(V)$ can be obtained as the quotient of the stabilizer ${\rm Stab}_{\Aut(V)}(\h)$ of the (fixed) Cartan subalgebra $\h$ in $\Aut(V)$ by the normal subgroup ${\rm Stab}_{\Inn(V)}(\h)={\rm Stab}_{\Aut(V)}(\h)\cap\Inn(V)$.
Since ${\rm Stab}_{\Aut(V)}(\h)$ preserves $V_L\otimes W$ and normalizes $S_\varphi^*$, the restriction of ${\rm Stab}_{\Aut(V)}(\h)$ to $V_L\otimes W$ is ${\rm Stab}_{\Aut(V_L\otimes W)}(\h)\cap{\rm Stab}_{\Aut(V_L\otimes W)}(S_\varphi)$.
By using $\Aut(V_L)$ and $S_\varphi$, we see that 
it acts on $(\mathcal{D}(L),q_L)$ as $\overline{O}(L)\cap \varphi^*(\overline{\Aut}(W))$,  
where $\mu_L:O(L)\to O(\mathcal{D}(L),q_L)$ (resp. $\mu_W:\Aut(W)\to O(\irr(W),q_W)$) is the canonical group homomorphism, $\overline{O}(L)$ (resp. $\overline{\Aut}(W)$) is the image of $\mu_L$ (resp. $\mu_W$) and $\varphi^*(\overline{\Aut}(W))=\varphi^{-1}\overline{\Aut}(W)\varphi$.
Then ${\rm Stab}_{\Aut(V)}(\h)$ acts on $\h$ as $\mu_L^{-1}(\overline{O}(L)\cap \varphi^*(\overline{\Aut}(W)))$.
Clearly, ${\rm Stab}_{\Inn(V)}(\h)$ acts on $\h$ as the Weyl group $W(V_1)$ of $V_1$.
Thus $\Out(V)\cong \mu_L^{-1}(\overline{O}(L)\cap \varphi^*(\overline{\Aut}(W)))/W(V_1)$ (Proposition \ref{P:OV}).
For each $\g$ in Schellekens' list with $0<\rank\g<24$, we describe $L=L_\g$ and $O(L)$ explicitly.
In addition, by using the structures of the groups $\overline{O}(L)$, $\varphi^*(\overline{\Aut}(W))(\cong\Aut(W))$ and $O(\mathcal{D}(L),q_L)$, we determine $\overline{O}(L)\cap \varphi^*(\overline{\Aut}(W))$, which gives the shape of $\Out(V)$.
In our calculations, we use the fact that except for the case $2C$, the index $|O(\irr(W),q_W):\overline{\Aut}(W)|$ is at most $4$, which implies that $\Out(V)$ has small index in $O(L)/W(V_1)$ (see \eqref{Eq:OV} and Lemma \ref{L:Out}). 

Our strategy for the uniqueness has been discussed in \cite{Ho2} (see Proposition \ref{P:Hoext}); we compute the number $|\varphi^*(\overline{\Aut}(W))\setminus O(\mathcal{D}(L),q_L)/\overline{O}(L)|$ of double cosets for given $W$ and $L$, which gives the number of holomorphic VOAs obtained as inequivalent simple current extensions of $V_L\otimes W$.
In fact, we verify that this number is $1$ if the conjugacy class is neither $2C$ nor $6G$, and compute the numbers for $2C$ and $6G$ by using the group structures of $O(L)$, $\Aut(W)$ and 
 $O(\mathcal{D}(L),q_L)$.
Then we obtain Theorem \ref{T:main13}.

The organization of the article is as follows. In Section 2, we review some basic notions for integral lattices and vertex operator algebras. In Section \ref{S:4}, we view holomorphic VOAs as simple current extensions of $V_L\otimes W$ and study some stabilizers. We will also describe the groups $\Out(V)$ and $K(V)$.
In Section \ref{sec:4}, we discuss the number of inequivalent simple current extensions of $V_L\otimes W$ that form holomorphic VOAs. 
In Section \ref{S:5}, for each $W$ mentioned in (d) and the semisimple Lie algebra $\g$ in Schellekens' list with $0<\rank \g<24$, we describe the orbit lattice $L=L_\g$ and determine the groups $K(V)$ and $\Out(V)$ explicitly.
In Appendix A, we describe the subgroup $\Out_1(V)$ of $\Out(V)$ and the quotient $\Out_2(V):=\Out(V)/\Out_1(V)$.

Some calculations on lattices and finite groups are done by MAGMA (\cite{MAGMA}).

\section*{Acknowledgments}
The authors would like to thank the reviewers for giving them careful suggestions and comments.
They also thank Brandon Rayhaun for useful discussion.

\subsection*{Conflict of interest statement} 
The authors have no competing interests to declare that are relevant to the content of this article.

\subsection*{Notations}

\begin{center}
	\begin{small}
		\begin{longtable}{ll} 
				    $2^{1+2n}_+$&
				       an extra special $2$-group of order $2^{1+2n}$ of plus type.
				        \\
				    $A.B$& a group $G$ that contains a normal subgroup $A$ with $G/A\cong B$.\\
		    $\overline{\Aut}(T)$  &  the subgroup of $O(\irr(T),q_T)$ induced by $\Aut(T)$, i.e, 
            $\overline{\Aut}(T)=\Imm\mu_{T}$.\\
            $\Aut_0(T)$ & the subgroup of $\Aut(T)$ which acts trivially on $\irr(T)$. \\
            $\Com_T(X)$& the commutant of a subset $X$ in a VOA $T$.\\
            $\mathcal{D}(H)$ & the discriminant group of an even lattice $H$, i.e.,  $\mathcal{D}(H)=H^*/H$. \\
            $\Inn(T)$ & the inner automorphism group of a VOA $T$ of CFT-type, i.e., \\
            			&the subgroup generated by  $\{\exp(a_{(0)})\mid a\in T_1\}$.\\
			$\irr(T)$ & the set of the isomorphism classes of irreducible modules over a VOA $T$.\\
						$K(T)$ &  the subgroup of $\Aut(T)$ which acts trivially on $T_1$ for a VOA $T$ of CFT-type.\\
			$H^g$& the fixed-point sublattice of a lattice $H$ by an isometry $g$.\\
			$H_g$& the coinvariant lattice of $g\in O(H)$, i.e., 
			$H_g = \{ x\in H \mid \langle x| H^g\rangle =0\}$. \\
			$L=L_\g$ & the orbit lattice associated with $\g$, i.e., $V_{L_\g}\cong \Com_V(\Com_V(\mathfrak{h}))$, where              
            $V$ is \\ &a holomorphic 
             VOA of $c=24$ with $V_1\cong\g$ and $\mathfrak{h}$ is a Cartan subalgebra of $\g$.\\
			$\mu_H$ &  
			the group homomorphism $\mu_H:O(H)\to O(\mathcal{D}(H),q_H)$ for an even lattice $H$.
				\\
			$\mu_T$ & 
				the group homomorphism $\mu_T:\Aut(T)\to O(\irr(T),q_T)$ for certain VOA $T$.
				\\
			$O(H)$  & the isometry group of a lattice $H$. \\
			$O( X,q)$ & the isometry group of a quadratic space $(X, q)$. \\
			$\OOut(\g)$ & $\OOut(\g)=\Aut(\g)/\Inn(\g)$, the group of outer automorphisms of a Lie algebra $\g$.  \\ 
			$\Out(T)$ & $\Out(T):=\Aut(T)/\Inn(T)$ for a VOA $T$ of CFT-type.\\
			$\overline{O}(H)$ & the subgroup of $O(\mathcal{D}(H),q_H)$ induced by $O(H)$, 
			i.e., $\overline{O}(H)=\Imm\mu_H$ \\
		    $O_0(H)$ &  the subgroup of $O(H)$ which acts trivially on $\mathcal{D}(H)$ for an even lattice $H$. \\ 
		    $p^n$& 
		     an elementary abelian $p$-group of order $p^n$.
		        \\
		    $p^{n+m}$& 
		       a $p$-group $G$ that contains  a normal subgroup $p^n$ with $G/p^n\cong p^m$.
		       \\
   			$P_\g$ & 
   			$P_\g=\bigoplus_{i=1}^s\frac{\sqrt{\ell}}{\sqrt{k_i}}Q^i\subset U_\g=\sqrt{\ell}L_\g^*$, where 
			$\g= \bigoplus_{i=1}^s \g_{i}$ is the direct sum of simple
					\\&
			 ideals,
			$k_i$ is the level of $\g_{i}$, $Q^i$ is the root lattice of $\g_{i}$ and $\ell$ is the level of $L_\g$.
   			\\
		    $q_H$ & the quadratic form on $\mathcal{D}(H)$, $q_H(v+H)=\langle v|v\rangle/2+\Z$.\\
			$q_T$ &  the quadratic form on $\irr(T)$ defined as conformal weights modulo $\Z$.\\ 
			$Q_\g$ &   $Q_\g=\bigoplus_{i=1}^s\sqrt{k_i}Q^i_{long}\subset L_\g$, where 
			$\g= \bigoplus_{i=1}^s \g_{i}$ is the direct sum of simple ideals,\\
			&$k_i$ is the level of $\g_{i}$ and $Q^i_{long}$ is the lattice spanned by long roots of $\g_{i}$. \\
			$R(H)$& the root system of a lattice $H$ (see Section \ref{S:lattice}).\\
			$\rho(M)$ & the conformal weight of an irreducible module $M$ over a VOA. \\
			$\Sym_n$& the symmetric group of degree $n$.\\
			${\rm Stab}_G(X)$ & the stabilizer of $X$ in a group $G$.\\
			$U=U_\g$  &    $U=\sqrt{\ell} L^*$, where $\ell$ is the level of the orbit lattice $L=L_\g$. \\

			$T^{\sigma}$& the set of fixed-points of an automorphism $\sigma$ of a VOA $T$.\\
			$W(R)$, $W(\g)$ & the Weyl group of a root system $R$ or a semisimple Lie algebra $\g$. \\
			$X_{n,k}$ & (the type of) a simple Lie algebra whose type is $X_n$ and level is $k$.\\
       \end{longtable}
	\end{small}
\end{center}

\section{Preliminary}

In this section, we review some basic terminology and notation for integral lattices and vertex operator algebras. 

\subsection{Lattices}\label{S:lattice}

By a \textit{lattice}, we mean a  free abelian group of finite rank with  a rational
valued, positive-definite symmetric bilinear form $\langle\ |\ \rangle$. 
A lattice $H$ is \textit{integral} if $\langle H| H\rangle \subset \Z$ and it is \textit{even} if 
$\langle x| x\rangle\in 2\Z$ for any $x\in H$. 
Note that an even lattice is integral.
Let $H^*$ denote the dual lattice of a lattice $H$, that is, $H^*= \{v\in 
\Q\otimes_\Z H \mid \langle v| H\rangle \subset \Z\}.$ 
If $H$ is integral, then $H\subset H^*$; $\mathcal{D}(H)$ denotes the discriminant group
$H^*/H$. 

An \textit{isometry} of a lattice $H$ is a linear isomorphism $g\in GL(\Q\otimes_\Z H)$ such that 
$g(H)=H$ and $\langle gx| gy\rangle=\langle x|y\rangle$ for all $x,y\in H$. 
Let $O(H)$ denote the group of all isometries of $H$, which we call the \emph{isometry group} of $H$.    
Note that $O(H)=O(H^*)$.

Let $H$ be an even lattice.
Let $q_H:\mathcal{D}(H)\to\Q/\Z$ denote the quadratic form on $\mathcal{D}(H)$ defined by $q_H(v+H)=\langle v|v\rangle/2+\Z$ for $v+H\in \mathcal{D}(H)$, and let $$\mu_H:O(H)\to O(\mathcal{D}(H),q_H)$$ denote the canonical group homomorphism, where $$O(\mathcal{D}(H),q_H)=\{g\in\Aut(\mathcal{D}(H))\mid q_H(gx)=q_H(x)\ \text{for all}\ x\in\mathcal{D}(H)\}.$$
The group $\overline{O}(H)$ denotes the subgroup of $O(\mathcal{D}(H),q_H)$ induced by $O(H)$, and $O_0(H)$ denotes the subgroup of $O(H)$ which acts trivially on $\mathcal{D}(H)$, that is, $$\overline{O}(H)=\Imm\mu_H,\qquad O_0(H)=\Ker\mu_H.$$

\begin{definition}\label{Def:coinv}
Let $H$ be a lattice and $g\in O(H)$. 
Let $H^g$ denote the fixed-point sublattice of $g$, that is, $H^g=\{x\in H\mid gx=x\}.$  
The \textit{coinvariant lattice} of $g$ is defined to be 
\[
H_g = \{ x\in H \mid \langle x| y\rangle =0 \text{ for all } y\in H^g\}.
\]
Clearly, the restriction of $g$ to $H_g$ is fixed-point free on $H_g$.
\end{definition}

Next we recall the definition of a root system from \cite{Hu}.
A subset $\Phi$ of $\R^n$ is called a \emph{root system} in $\R^n$ if $\Phi$ satisfies (R1)--(R4) below: 
\begin{enumerate}[(R1)]
\item $|\Phi|<\infty$ and $\Phi$ spans $\R^n$; 
\item If $\alpha\in \Phi$, then $\Z\alpha\cap\Phi=\{\pm\alpha\}$; 
\item If $\alpha\in\Phi$, then the reflection $\sigma_\alpha:\beta\mapsto \beta-2(\langle \beta|\alpha\rangle/\langle\alpha|\alpha\rangle)\alpha$ leaves $\Phi$ invariant; 
\item If $\alpha,\beta\in\Phi$, then $\langle\beta|\alpha\rangle/\langle\alpha|\alpha\rangle\in\Z$. 
\end{enumerate}
The \emph{root lattice} $L_\Phi$ of $\Phi$ is the lattice spanned by roots.
If $\Phi$ is irreducible, of type $A_n$, $D_n$ or $E_n$ and $\langle\alpha|\alpha\rangle=2$ for all $\alpha\in\Phi$, then we often denote $L_\Phi$ just by $\Phi$.

Let $L$ be a positive-definite rational lattice.
An element $\alpha\in L$ is \emph{primitive} if $L/\Z\alpha$ has no torsion.
A primitive element $\alpha\in L$ is called a \emph{root} of $L$ if the reflection $\sigma_\alpha$ in the ambient space of $L$ is in $O(L)$.
The set $R(L)$ of roots is an abstract root system in the ambient space of the sublattice $L_{R(L)}$ of $L$ spanned by $R(L)$.
Hence the general theory of root system applies to $R(L)$ and $R(L)$ decomposes into irreducible components of type $A_n$, $B_n$, $C_n$, $D_n$ or $G_2$, $F_4$, $E_6$, $E_7$, $E_8$.

For $\ell\in \Z_{>0}$ and a lattice $H$, we denote $\sqrt{\ell} H= 
\{ \sqrt{\ell}x\mid x\in H\}$.
The \textit{level} of an even lattice $H$ is defined to be the smallest positive integer 
$\ell$ such that $\sqrt{\ell}H^*$ is again even. 
The following can be obtained from \cite[Propositions 2.1 and 2.2]{Sc06}.

\begin{lemma}\label{RL}
Let $H$ be an even lattice of level $\ell$.
\begin{enumerate}[{\rm (1)}]
\item  Let $\alpha$ be a root of $H$ with $\langle\alpha|\alpha\rangle=2k$.
Then $k\mid \ell$ and $\alpha\in H\cap k H^*$.
\item Assume that $\ell$ is prime.
Then  
\begin{equation*}
R(H) =\{ v\in H \mid \langle v|v\rangle=2\} \cup \{ v\in \ell H^* \mid \langle v|v\rangle=2\ell\}.\label{Eq:RL}
\end{equation*}
\end{enumerate}
\end{lemma}

For a root system $\Phi$, the \emph{Weyl group} $W(\Phi)$ is the subgroup of $O(L_\Phi)$ generated by reflections associated with elements in $\Phi$.
The following lemma is well-known:

\begin{lemma}\label{Lem:W}
There are isomorphisms of the Weyl groups of root systems and the isometry groups of the root lattices:
\begin{align*}
W(B_4)&\cong W(C_4)\cong W(D_4).2,\quad W(B_n)\cong W(C_n)\cong O(D_n),\quad (n\ge2, n\neq4),\\
W(F_4)&\cong O(D_4)\cong W(D_4).\Sym_3,\quad W(G_2)\cong O(A_2),
\end{align*}
where $D_2= A_1^2$ and $D_3= A_3$.
\end{lemma}

\subsection{Vertex operator algebras}\label{preVOA}
Throughout this article, all VOAs are defined over the field $\C$ of complex numbers. 

A \emph{vertex operator algebra} (VOA) $(T,Y,\1,\omega)$ is a $\Z$-graded vector space $T=\bigoplus_{m\in\Z}T_m$ over the complex field $\C$ equipped with a linear map
$$Y(a,z)=\sum_{i\in\Z}a_{(i)}z^{-i-1}\in ({\rm End}\ (T))[[z,z^{-1}]],\quad a\in T,$$
the \emph{vacuum vector} $\1\in T_0$ and the \emph{conformal vector} $\omega\in T_2$
satisfying certain axioms (\cite{Bo,FLM}). 
Note that the operators $L(m)=\omega_{(m+1)}$, $m\in \Z$, satisfy the Virasoro relation:
$$[L{(m)},L{(n)}]=(m-n)L{(m+n)}+\frac{1}{12}(m^3-m)\delta_{m+n,0}c\ {\rm id}_T,$$
where $c\in\C$ is called the \emph{central charge} of $T$, and $L(0)$ acts by the multiplication of scalar $m$ on $T_m$.

A linear automorphism $\sigma$ of a VOA $T$ is called a (VOA) \emph{automorphism} of $T$ if $$ \sigma\omega=\omega\quad {\rm and}\quad \sigma Y(v,z)=Y(\sigma v,z)\sigma\quad \text{ for all } v\in T.$$
The group of all (VOA) automorphisms of $T$ is denoted by $\Aut (T)$. 

A \emph{vertex operator subalgebra} (or a \emph{subVOA}) of a VOA $T$ is a graded subspace of
$T$ which has a structure of a VOA such that the operations and its grading
agree with the restriction of those of $T$ and  they share the vacuum vector.
In addition, if they also share the conformal vector, then the subVOA is said to be \emph{full}.
For an automorphism $\sigma$ of a VOA $T$, let $T^\sigma$ denote the fixed-point set of $\sigma$, i.e.,
$$T^\sigma=\{v\in T\mid \sigma v=v\},$$
which is a full subVOA of $T$.
For a subset $X$ of a VOA $T$, the \emph{commutant} ${\rm Com}_T(X)$ of $X$ in $T$ is the subalgebra of $T$ which commutes with $X$ (\cite{FZ}).
Note that the double commutant $\Com_T(\Com_T(X))$ contains $X$.

Let $M=\bigoplus_{m\in\C} M_m$ be a module over a VOA $T$ (see \cite{FHL} for the definition).
If $M$ is irreducible, then there exists unique $\rho(M)\in\C$ such that $M=\bigoplus_{m\in\Z_{\geq 0}}M_{\rho(M)+m}$ and $M_{\rho(M)}\neq0$.
The number $\rho(M)$ is called the \emph{conformal weight} of $M$.
Let $\irr(T)$ denote the set of isomorphism classes of irreducible $T$-modules.
We often identify an irreducible module with its isomorphism class without confusion.

A VOA is said to be  \emph{rational} if the admissible module category is semisimple.
(See \cite{DLM2} for the definition of admissible modules.)
A rational VOA is said to be \emph{holomorphic} if it itself is the only irreducible module up to isomorphism.
A VOA $T$ is \emph{of CFT-type} if $T_0=\C\1$ (note that $T_i=0$ for all $i<0$ if $T_0=\C\1$), and is \emph{$C_2$-cofinite} if the co-dimension in $T$ of the subspace spanned by $\{u_{(-2)}v\mid u,v\in T\}$ is finite.
If $T$ is rational and $C_2$-cofinite, then $\rho(M)\in\Q$ for any $M\in\irr(T)$ (\cite[Theorem 1.1]{DLM2}).
A module over a VOA is said to be \emph{self-contragredient} if it is isomorphic to its contragredient module (see \cite{FHL}).
A VOA is said to be \emph{strongly regular} if it is rational, $C_2$-cofinite, self-contragredient and of CFT-type.
Note that a strongly regular VOA is simple.
A simple VOA $T$ of CFT-type is said to satisfy the \emph{positivity condition} if $\rho(M)\in \R_{>0}$ for all $M\in\irr(T)$ with $M\not\cong T$.

Let $T$ be a VOA and let $M$ be a $T$-module.
For $\sigma\in\Aut(T)$, let $M\circ \sigma$ denote the \emph{$\sigma$-conjugate module}, i.e., $M\circ \sigma=M$ as a vector 
space and its vertex operator is $Y_{M\circ \sigma}( u,z)= Y_M(\sigma u,z)$ for $u\in T$.
If $M$ is irreducible, then so is $M\circ \sigma$.
Hence $\Aut(T)$ acts on $\irr(T)$ as follows: for $\sigma\in \Aut(T)$, $M\mapsto M\circ \sigma$.
Note that $\rho(M)=\rho(M\circ \sigma)$ for $\sigma\in\Aut(T)$ and $M\in\irr(T)$.

Let $T$ be a strongly regular VOA.
Then the \emph{fusion products} $\boxtimes$ are defined on irreducible $T$-modules (\cite{HL}).
Note that the action of $\Aut(T)$ on $\irr(T)$ above also preserves the fusion products.
An irreducible $T$-module $M^1$ is called a \emph{simple current module} if for any irreducible $T$-module $M^2$, the fusion product $M^1\boxtimes M^2$ is also an irreducible $T$-module.
If all irreducible $T$-modules are simple current modules, then $\irr(T)$ has an abelian group structure under the fusion products; in this case, we say that $T$ has \emph{group-like fusion}. 

\begin{theorem}[{\cite[Theorem 3.4, Proposition 3.5]{EMS}}]\label{EMS1} Let $T$ be a strongly regular VOA.
Assume that $T$ has group-like fusion and satisfies the positivity condition.
Let $$q_T:\irr(T)\to\Q/\Z,\quad M\mapsto\rho(M)\mod \Z.$$
Then $q_T$ is a quadratic form on the abelian group $\irr(T)$ and the associated bilinear form is non-degenerate.
\end{theorem}

\begin{remark} We call a finite abelian group with a quadratic form \emph{a quadratic space}.
\end{remark}

Let $T$ be a strongly regular VOA satisfying the assumption of Theorem \ref{EMS1}.
Then, we obtain the canonical group homomorphism 
\begin{equation}
\mu_T:\Aut(T)\to O(\irr(T),q_T),\label{Eq:Orthogonal}
\end{equation}
where $O(\irr(T),q_T)=\{f\in\Aut(\irr(T))\mid q_T(W)=q_T(f(W))\ \text{for all}\ W\in\irr(T)\}$ is the orthogonal group of the quadratic space $(\irr(T),q_T)$.
The group $\overline{\Aut}(T)$ denotes the subgroup of $O(\irr(T),q_T)$ induced by $\Aut(T)$, and $\Aut_0(T)$ denotes the subgroup of $\Aut(T)$ which acts trivially on $\irr(T)$, that is, $$\overline{\Aut}(T)=\Imm\mu_{T},\qquad \Aut_0(T)=\Ker\mu_T.$$

Let $T^0$ be a strongly regular VOA.
Let $\{T^\alpha\mid \alpha\in D\}$ be a set of inequivalent irreducible $T^0$-modules indexed by a finite abelian group $D$.
A simple VOA $T_D=\bigoplus_{\alpha\in D}T^\alpha$ is called a {\it simple current extension} of $T^0$ if every $T^\alpha$ is a simple current module.
Note that $T^\alpha\boxtimes_{T^0}T^\beta\cong T^{\alpha+\beta}$ and that the simple VOA structure of $T_D$ is uniquely determined by its $T^0$-module structure, up to isomorphism (\cite[Proposition 5.3]{DM04b}).
Two simple current extensions $T_{D}$ and $T_{E}$ of $T^0$ are \emph{equivalent} if there exists an isomorphism $\sigma:T_D\to T_E$ such that $\sigma(T^0)=T^0$, equivalently, 
there exists $\tau\in\Aut(T^0)$ such that $T_D\cong T_E\circ\tau$ as $T^0$-modules.

\subsection{Automorphisms of lattice VOAs}
Let $H$ be an even lattice and let $V_H$ be the lattice VOA associated with $H$ (see \cite{FLM} for detail).
It is well-known (\cite{Do}) that $V_H$ is strongly regular, has group-like fusion and satisfies the positivity condition.
In addition, $\irr(V_H)=\{V_{\lambda+H}\mid \lambda+H\in\mathcal{D}(H)\}$ and $(\irr(V_H),q_{V_H})\cong (\mathcal{D}(H),q_H)$ as quadratic spaces (see \cite{Do}).

Let $\hat{H}=\{\pm e^\alpha\mid \alpha\in H\}$ be a central extension of $H$ by $\{\pm1\}$ satisfying $e^\alpha e^\beta=(-1)^{\langle\alpha|\beta\rangle}e^\beta e^\alpha$ for $\alpha,\beta\in H$.
Note that such a central extension is unique up to isomorphism.
Let $\Aut(\hat{H})$ be the set of all automorphisms of $\hat{H}$.
For $\varphi\in \hat{H}$, we define the element $\iota({\varphi})\in \Aut(H)$ by $\varphi(e^\alpha)\in\{\pm e^{\iota({\varphi})(\alpha)}\}, \alpha\in H$.
Set $O(\hat{H})=\{\varphi\in\Aut(\hat{H})\mid \iota({\varphi})\in O(H)\}$.
It was proved in \cite[Proposition 5.4.1]{FLM} that there exists an exact sequence:
\begin{equation}
1\to \mathrm{Hom}(H,\Z_2)\to O(\hat{H})\ \xrightarrow{\iota}\ O(H)\to1.\label{Eq:exact}
\end{equation}
We also identify $O(\hat{H})$ as a subgroup of $\Aut(V_H)$ as in \cite[Section 2.4]{DN99}.
Note that $\mathrm{Hom}(H,\Z_2)=\{\exp(2\pi\sqrt{-1}\alpha_{(0)})\mid \alpha\in (H^*/2)/H^*\}$ in $\Aut(V_H)$.
 
For $g\in O(H)$, an element $\tau\in O(\hat{H})$ with $\iota(\tau)=g$ is called a \emph{standard lift} of $g$ if $\tau$ acts trivially on the subVOA $V_{H^g}$.
Note that a standard lift of $g$ always exists and standard lifts of $g$ are conjugate in $\Aut(V_H)$ (\cite[Proposition 7.1]{EMS} or \cite[Proposition 4.6]{LS6}); we often denote a standard lift of $g$ by $\hat{g}$.
If $g$ is fixed-point free on $H$, then we have $|\hat{g}|=|g|$ (\cite[Proposition 7.4]{EMS}).

Recall from \cite[Theorem 2.1]{DN99} that 
\begin{equation}
\Aut(V_{H})=\Inn(V_{H})O(\hat{H}).\label{Eq:DN}
\end{equation}
Set $\h={\rm Span}_\C\{ h(-1)\1\mid h\in H\}$.
Then $\h$ is a Cartan subalgebra of the reductive Lie algebra $(V_H)_1$.
By \cite[Lemmas 2.3 and 2.5]{DN99}, we have 
\begin{equation}
\{\sigma\in\Aut(V_H)\mid \sigma=id\ \text{on}\ \h\}=\{\exp(a_{(0)})\mid a\in\h\}\label{Eq:idh}
\end{equation}
and 
\begin{equation}
{\rm Stab}_{\Aut(V_H)}(\h)=\{\sigma\in\Aut(V_H)\mid \sigma(\h)=\h\}=\{\exp(a_{(0)})\mid a\in\h\} O(\hat{H}).\label{Eq:stabh}
\end{equation}
It follows from \eqref{Eq:idh}, \eqref{Eq:stabh} and 
$\ker\iota\subset \{\exp(a_{(0)})\mid a\in\h\}$ (cf. \eqref{Eq:exact}) 
that
\begin{equation}
{\rm Stab}_{\Aut(V_H)}(\h)_{|\h}\cong{\rm Stab}_{\Aut(V_H)}(\h)/\{\exp(a_{(0)})\mid a\in\h\}\cong O(H).\label{Eq:stabhVL}
\end{equation}
The explicit action of $\Aut(V_{H})$ on $\irr(V_{H})$ via the conjugation in Section \ref{preVOA} is well-known (cf. {\cite[Lemma 2.11]{LS6} and \cite[Proposition 2.9]{Sh04}}):

\begin{lemma}\label{L:innM}
\begin{enumerate}[{\rm (1)}]
\item For $\sigma\in\Inn(V_H)$ and $M\in \irr(V_H)$, we have $M\circ \sigma\cong M$, that is, $\Inn(V_H)\subset \Aut_0(V_H)$.
\item For $\sigma\in O(\hat{H})$, we have $ V_{\lambda+H}\circ \sigma \cong V_{(\iota \sigma)^{-1}(\lambda)+H}$ for any 
$\lambda+H\in \mathcal{D}(H)$.
\end{enumerate}
\end{lemma}

By \eqref{Eq:exact}, \eqref{Eq:DN} and Lemma \ref{L:innM}, we have the following.
\begin{lemma}\label{L:XL}
$\Aut_0(V_H)= \Inn(V_H) \iota^{-1}(O_0(H))$ and $\overline{\Aut}(V_H)\cong O(H)/O_0(H)\cong\overline{O}(H).$   
\end{lemma}

\section{Holomorphic VOAs of central charge $24$ as simple current extensions}\label{S:4}
Let $V$ be a (strongly regular) holomorphic VOA of central charge $24$.
By \cite{DM04a,DM04b}, $V$ satisfies one of the following:
\begin{enumerate}[(i)]
\item $V_1=0$;
\item $V$ is isomorphic to a Niemeier lattice VOA;
\item $V_1$ is a semisimple Lie algebra whose Lie rank $\rank V_1$ is less than $24$.
\end{enumerate}
Note that in (ii) and (iii), the subVOA generated by $V_1$ is a full subVOA (\cite[Proposition 4.1]{DM04a}).
In this section, we assume (iii), i.e., 
$0<\rank V_1<24$,
and explain how to determine $K(V)$ and $\Out(V)$.

\begin{remark}\label{R:KOknown}
It is conjectured that if (i) holds, then $V$ is isomorphic to the moonshine VOA $V^\natural$ (\cite{FLM}).
Note that $K(V^\natural)(=\Aut(V^\natural))$ is the Monster simple group and $\Inn(V^\natural)=1$, which shows $K(V^\natural)\not\subset\Inn(V^\natural)$ and $\Out(V^\natural)=\Aut(V^\natural)/\Inn(V^\natural)\cong\Aut(V^\natural)$.

If (ii) holds, then $K(V)$ and $\Out(V)$ are easily determined by \eqref{Eq:DN}; indeed, $K(V_\Lambda)\cong \C^{24}$ and $\Out(V_\Lambda)=O(\Lambda)$ for the Leech lattice $\Lambda$ and $K(V_N)\cong N/Q$ and 
$\Out(V_N)\cong O(N)/W(Q)$ for a Niemeier lattice $N$ with the root lattice $Q\neq\{0\}$.
By \eqref{Eq:idh}, $K(V_N)\subset\Inn(V_N)$ for any Niemeier lattice $N$.
\end{remark}

\subsection{Commutant of a Cartan subalgebra}\label{S:31}
Let $V$ be a holomorphic VOA of central charge $24$ with $0<\rank V_1<24$.
Set $\g=V_1$ and let $\h$ be a Cartan subalgebra of $V_1$.
Set $W=\Com_V(\h)$.
Then $W_1=0$.
Recall from \cite[Corollary 5.8]{DM06} that the double commutant of a Cartan subalgebra in a simple affine VOA at positive level is a lattice VOA.
Since the subVOA generated by $V_1$ is a tensor product of simple affine VOAs at positive level (\cite[Theorem 1.1]{DM06}), the double commutant $\Com_{V}(\Com_V(\h))$ contains a lattice VOA as a full subVOA; there exists an even lattice $L$ such that $$\Com_{V}(\Com_V(\h))\cong V_L.$$
In fact, $L$ is uniquely determined by the Lie algebra structure of $\g$, which will be verified by the explicit description of $L$ in Section \ref{S:5} (cf.\ \cite{Ho2,ELMS}); $L=L_\g$ is called the \emph{orbit lattice} in \cite{Ho2}.
Hence $V$ contains $V_L\otimes W$ as a full subVOA, which shows  
\begin{equation}
\rank L+c_W=24,\label{Eq:quad0}
\end{equation}
where $c_W$ is the central charge of $W$.
Note that the injective map from $V_L\otimes W$ to $V$ is given by $a\otimes b\mapsto a_{(-1)}b$ for $a\in V_L$ and $b\in W$.
By \cite{Mi15,CM} and \cite[Section 4.3]{CKLR19}, $W$ is also strongly regular.
In addition, by \cite[Lemma 5.2]{ELMS}, $W$ satisfies the positivity condition;
indeed, $W$ contains a full subVOA isomorphic to the tensor product of parafermion VOAs (\cite{DR}), which satisfies the positivity condition.

It then follows from \cite{Lin,CKM} that $W$ has group-like fusion and 
\begin{equation}
(\irr(V_L),q_{V_L})\cong(\irr(W),-q_W)\label{Eq:quad}
\end{equation}
as quadratic spaces.
Note that $O(\irr(W),q_W)=O(\irr(W),-q_W)$ as groups.
The VOA $W$ was essentially identified in \cite[Theorem 4.7]{Ho2} (cf.\ \cite[Theorem 4.2]{HM}) as follows; note that the necessary assumptions are confirmed in \cite{Lam20}.

\begin{theorem}\label{T:Ho} 
The VOA $W$ is isomorphic to the orbifold VOA $V_{\Lambda_g}^{\hat{g}}$
for an isometry $g$ of the Leech lattice $\Lambda$, where $g$ belongs to one of $10$ conjugacy classes $2A,2C,3B$, $4C,5B,6F$, $6G,7B,8E$ and $10F$, and $\hat{g}$ is a (standard) lift of $g_{|\Lambda_g}\in O(\Lambda_g)$.
In addition, the conjugacy class $g$ is uniquely determined by the structure of $V_1$.
\end{theorem}

\begin{remark} Theorem \ref{T:Ho} can also be proved by using the fact that any holomorphic VOA of central charge $24$ is constructed from the Leech lattice VOA by a cyclic orbifold construction (\cite[Theorem 6.3]{ELMS}).
\end{remark}

\begin{theorem}\cite{GO+102,Sh04,Lam18b,LamCFT,BLS}\label{T:inj}
Let $g\in O(\Lambda)$ whose conjugacy class is one of $10$ cases in Theorem \ref{T:Ho}.
Then the automorphism group of $W\cong V_{\Lambda_g}^{\hat{g}}$ has the shape as in Table \ref{Table:main} (see \cite{Wi} for the notation of classical groups).
In addition, the group homomorphism $\mu_W$ in \eqref{Eq:Orthogonal} 
is injective and the index of $\overline{\Aut}(W)(\cong\Aut(W))$ in $O(\irr(W),q_W)$ is given as in Table \ref{Table:main}.
\end{theorem}
\begin{remark} 
The shapes of some groups in Table \ref{Table:main} are recalculated by MAGMA; they are more precise than the original shapes in the references.
We adopt the genus symbol $B,C,\dots,K$ of $(\irr(W),-q_W)$ and quadratic space structures from \cite[Table 4]{Ho2}.

\end{remark}

{\tiny
\begin{longtable}{|c|c|c|c|c|c|c|}
\caption{VOAs $W= V_{\Lambda_g}^{\hat{g}}$ for $g\in O(\Lambda)$}\label{Table:main}
\\ \hline 
Genus&Class& ${\rm rank}\Lambda_g$&$(\irr(W),-q_W)$&$\Aut(W)(\cong\overline{\Aut}(W))$&$O(\irr(W),q_W)$&index\\ \hline
$B$&$2A$&$8$&$2^{+10}_{\rm{I\hspace{-.01em}I}}$&$\GO_{10}^+(2)$&$\GO_{10}^+(2)$&$1$\\
$C$&$3B$&$12$&$3^{-8}$&$\PO_8^-(3).2$&$\GO_8^-(3)$&$2$\\
$D$&$2C$&$12$&$2^{-10}_{\rm{I\hspace{-.01em}I}}4^{-2}_{\rm{I\hspace{-.01em}I}}$&
$2^{1+20}_+.(\Sym_{12}\times\Sym_3)$
&$2^{1+20}_+.(\GO_{10}^-(2)\times\Sym_3)$&$2^{11}\cdot 3\cdot17$\\
$E$&$4C$&$14$& $2^{+2}_{2}4^{+6}_{\rm{I\hspace{-.01em}I}}$&$2^{21}.\GO_7(2)$&$2^{22}.\GO_7(2)$&$2$ \\
$F$&$5B$&$16$&$5^{+6}$&$2.\PO_6^+(5).2$&$\GO_6^+(5)$&$2$\\
$G$&$6E$&$16$&$2^{+6}_{\rm{I\hspace{-.01em}I}}3^{-6}$&$\GO_6^+(2)\times \GO_6^+(3)$&$\GO_6^+(2)\times \GO_6^+(3)$&$1$\\ 
$H$&$7B$&$18$& $7^{-5}$&$\PO_5(7).2$&$\GO_5(7)$&$2$\\
$I$&$8E$&$18$& $2^{+1}_{5}4^{+1}_{1}8^{+4}_{\rm{I\hspace{-.01em}I}}$&$2^{11+9}.\Sym_6$&$2^{12+9}.\Sym_6$&$2$\\
$J$&$6G$&$18$&$2^{+4}_{\rm{I\hspace{-.01em}I}}4^{-2}_{\rm{I\hspace{-.01em}I}}3^{+5}$&$ 2^{1+8}_+{:}(\Sym_3^3) \times \PO_5(3).2.$&$2^{1+8}_+{:}(\GO_4^+(2)\times\Sym_3)\times \GO_5(3)$&$4$\\
$K$&$10F$&$20$&$2^{-2}_{\rm{I\hspace{-.01em}I}}4^{-2}_{\rm{I\hspace{-.01em}I}}5^{+4}$&$2^{1+4}_+{:}(2\times \Sym_3)\times \GO_4^+(5)$&$2^{1+4}_+{:}(\Sym_3\times \Sym_3)\times \GO_4^+(5)$&$3$\\
\hline 
\end{longtable}
}

The following properties of $\overline{\Aut}(W)(\cong\Aut(W))$ will be used later.

\begin{lemma}\label{Lem:center}
Assume that $g\in O(\Lambda)$ belongs to one of $10$ conjugacy classes in Theorem \ref{T:Ho}.
Set $W=V_{\Lambda_g}^{\hat{g}}$.
\begin{enumerate}[{\rm (1)}]
\item If the conjugacy class of $g$ is neither $2C$, $6G$  nor $10F$, then $\overline{\Aut}(W)$ is a normal subgroup of $O(\irr(W),q_W)$. 
\item If the conjugacy class of $g$ is $3B,4C,6G,7B$ or $8E$, 
then $\overline{\Aut}(W)$ does not contain the $-1$-isometry of the abelian group $\irr(W)$.
\end{enumerate}
\end{lemma}
\begin{proof}
(1) is obvious from the indexes in Table \ref{Table:main}.

Assume that $\overline{\Aut}(W)$ contains the $-1$-isometry $\sigma$; we view $\sigma$ as an element of $\Aut(W)$.
Then for any $M\in\irr(W)$, $M\circ \sigma$ is the contragredient module $M'$ of $M$.
Recall that the fusion products in $\irr(W)$ are determined in \cite{Lam20}.
In particular, $V_{\lambda+\Lambda_g}'\cong V_{-\lambda+\Lambda_g}$ as $W$-modules for any $\lambda+\Lambda_g\in\mathcal{D}(\Lambda_g)$.
Then, $V_{\Lambda_g}\circ\sigma\cong V_{\Lambda_g}$, which shows that $\sigma$ can be lifted to an automorphism of $V_{\Lambda_g}$ (\cite[Theorem 3.3] {Sh04}); we fix such an automorphism of $V_{\Lambda_g}$ and use the same symbol $\sigma$.
In addition, $V_{\lambda+\Lambda_g}\circ \sigma\cong V_{-\lambda+\Lambda_g}$ as $V_{\Lambda_g}$-modules for any $\lambda+\Lambda_g\in\mathcal{D}(\Lambda_g)$.
By \eqref{Eq:exact}, \eqref{Eq:DN} and Lemma \ref{L:innM}, there exists $f\in O(\Lambda_g)$ of order $2$ such that $\sigma\in \Inn(V_{\Lambda_g})\iota^{-1}(f)$ and $f=-1$ on $\mathcal{D}(\Lambda_g)$.
By the fusion products in $\irr(W)$, the $\sigma$-conjugate modules of irreducible $\hat{g}^i$-twisted $V_{\Lambda_g}$-modules are irreducible $\hat{g}^{-i}$-twisted $V_{\Lambda_g}$-modules.
Hence $f gf^{-1}=g^{-1}$.
Then $-f$ is an element in $O_0(\Lambda_g)$ of order $2$ and $(-f)g(-f)^{-1}=g^{-1}$.
It follows from $\Lambda^*=\Lambda$ that for any element $\lambda+\Lambda_g\in \mathcal{D}(\Lambda_g)$ there exists $\xi+\Lambda^g\in\mathcal{D}(\Lambda^g)$ such that $(\lambda+\Lambda_g,\xi+\Lambda^g)$ appears in $\Lambda/(\Lambda^g\oplus\Lambda_g)$.
Since $g_{|\Lambda^g}$ act trivially on $\mathcal{D}(\Lambda^g)$ and $g\in O(\Lambda)$, we see that $g$ preserves every element in $\Lambda/(\Lambda^g\oplus\Lambda_g)$.
Hence $g\in O_0(\Lambda_g)$.
Thus, $O_0(\Lambda_g)$ contains the subgroup $\langle f,g\rangle$ isomorphic to the dihedral group of order $2|g|$.

By using MAGMA, one can verify the following: if the conjugacy class of $g$ is $3B$, $4C$ or $6G$, then $O_0(\Lambda_g)$ is the cyclic group $\langle g\rangle$; if the conjugacy class of $g$ is $7B$, then $O_0(\Lambda_g)$ has order $21$; if the conjugacy class of $g$ is $8E$, then $O_0(\Lambda_g)$ has order $16$ but it is not the dihedral group of order $16$.
Hence we obtain (2).
\end{proof}

\begin{remark}\label{R:5Bb} If the conjugacy class of $g$ is $5B$, then $O_0(\Lambda_g)$ is a dihedral group of order $10$ (\cite{GL5A}); in fact, $\Aut(W)$ contains the $-1$-isometry (\cite{Lam18b}).
\end{remark}

\begin{lemma}\label{L:level}
Let $g\in O(\Lambda)$ whose conjugacy class is one of the $10$ cases in Theorem \ref{T:Ho}.
Set $\ell=2|g|$ if the conjugacy class of $g$ is $2C$, $6G$ or $10F$, and set $\ell=|g|$ otherwise.
Set $W=V_{\Lambda_g}^{\hat{g}}$.
Let $H$ be an even lattice satisfying $(\mathcal{D}(H),q_H)\cong (\irr(W),-q_W)$ as quadratic spaces.
Then $H$ has level $\ell$.
Moreover, if the conjugacy class of $g$ is $2A$, $3B$, $5B$, $6E$ or $7B$ and the rank of $H$ is $24-\rank\Lambda_g$, then $\sqrt{\ell}H^*$ also has level $\ell$.

\end{lemma}
\begin{proof}
By the classification of irreducible $W$-modules (see \cite{Lam20}), one can see that $\ell$ is the minimal positive integer such that $q_W(\irr(W))\subset(1/\ell)\Z_{\ge0}$.
It then follows from $(\mathcal{D}(H),q_H)\cong (\irr(W),-q_W)$ that $\ell$ is also the minimal positive integer satisfying $q_H(\mathcal{D}(H))\subset(1/\ell)\Z_{\ge0}$.
Hence $\ell$ is the minimal positive integer so that $\sqrt{\ell}H^*$ is even, and $H$ has level $\ell$.

Assume that the conjugacy class of $g$ is $2A$, $3B$, $5B$, $6E$ or $7B$ and the rank of $H$ is $24-\rank\Lambda_g$.
Since $\sqrt{\ell}(\sqrt{\ell}H^*)^*=H$, the latter assertion follows from the fact that $(1/\sqrt{n})H$ is not even if $n\in \Z_{>1}$.
Indeed, if $(1/\sqrt{n})H$ is even for $n\in\Z_{\ge1}$, then $(1/\sqrt{n})H\subset \sqrt{n}H^*$, and hence $n^{\rank H}$ divides $|H^*/H|=|\irr(W)|$.
By Table \ref{Table:main}, the only possibility is $n=1$.
\end{proof}

\begin{remark}\label{R:ell} In Lemma \ref{L:level}, $\ell$ is equal to $|\hat{g}|$ for the (standard) lift $\hat{g}\in O(\hat{\Lambda})$ of $g$ (cf.\ \cite[Proposition 7.4]{EMS}).
\end{remark}

\subsection{The group $K(V)$}
Let $V$ be a holomorphic VOA of central charge $24$ with $0<\rank V_1<24$.
Let $\h$ be a Cartan subalgebra of $V_1$.
Set $W=\Com_V(\h)$ and $V_L=\Com_V(W)$ as in Section \ref{S:31}.
In this subsection, we describe the group $K(V)$, defined in the introduction, in terms of $V_1$ and $L$.

Recall that $V_L\otimes W$ has group-like fusion.
Hence $V$ is a simple current extension of $V_L\otimes W$.
Since $V$ is holomorphic, for any irreducible $V_L$-module $V_{\lambda+L}$, there exists a unique irreducible $W$-module $X$ such that $V_{\lambda+L}\otimes X$ appears as an irreducible $V_L\otimes W$-submodule of $V$ with multiplicity one; let $\varphi$ be the bijection from $\mathcal{D}(L)$ to $\irr(W)$ defined by the following decomposition of $V$ as a $V_L\otimes W$-module:
\begin{equation}
V\cong\bigoplus_{\lambda+L\in\mathcal{D}(L)}V_{\lambda+L}\otimes \varphi({\lambda+L})\label{Eq:V}.
\end{equation}
Then $\varphi$ is a group isomorphism and $\rho(V_{\lambda+L})+\rho(\varphi({\lambda+L}))\in\Z$, which shows that $\varphi$ is an isometry of quadratic spaces from $(\mathcal{D}(L),q_{L})$ to $(\irr(W),-q_W)$.
Set 
\begin{equation}
S_\varphi=\{(V_{\lambda+L},\varphi(\lambda+L))\mid \lambda+L\in\mathcal{D}({L})\}\subset\irr(V_L)\times\irr(W).\label{Eq:S}
\end{equation}
Since $V$ is holomorphic, $S_\varphi$ is a maximal totally isotropic subspace of $(\irr(V_L),q_{V_L})\oplus(\irr(W),q_W)$.
Here a vector is isotropic if the value of the form is zero and a totally isotropic subspace is a subspace consisting of isotropic vectors.
Note that $S_\varphi\cong \mathcal{D}(L)$ as groups.
We then view $V$ as a simple current extension of $V_L\otimes W$ graded by $S_\varphi$; $V=\bigoplus_{M\in S_\varphi}M$.
Here $(V_{\lambda+L},\varphi(\lambda+L))\in S_\varphi$ is regarded as an irreducible $V_L\otimes W$-module $V_{\lambda+L}\otimes\varphi(\lambda+L)$. 
Hence the dual $S_\varphi^*={\rm Hom}(S_\varphi,\C^\times)$ of $S_\varphi$ acts faithfully on $V$ as an automorphism group.
More precisely, by \eqref{Eq:S}, we have 
\begin{equation}
S_\varphi^*=\{\exp(2\pi\sqrt{-1}v_{(0)})\mid v+L\in \mathcal{D}(L)\}.\label{Eq:S*}
\end{equation}
In addition, by \cite[Theorem 3.3]{Sh04}, we obtain
\begin{align}
S_\varphi^*&=\{\sigma\in\Aut(V)\mid \sigma=id\ {\rm on}\ V_L\otimes W\}\label{Eq:ns0}.
\end{align}
\begin{proposition}\label{L:idh} $\{\sigma\in {\Aut(V)}\mid \sigma=id\ \text{on}\ \h\}=\{\exp(a_{(0)})\mid a\in\h\}.$ 
In particular, $K(V)\subset \Inn(V)$.
\end{proposition}
\begin{proof}
Clearly, $\{\sigma\in \Aut(V)\mid \sigma=id\ \text{on}\ \h\}\supset\{\exp(a_{(0)})\mid a\in\h\}.$ 

Let $\sigma\in\Aut(V)$ such that $\sigma=id$ on $\h$.
Then $\sigma$ preserves the commutant and the double commutant of $\h$, that is, 
$\sigma$ preserves both $V_L$ and $W$.
Since $\sigma_{|V_L}$ acts trivially on $\h\subset(V_L)_1$, we have $\sigma_{|V_L}\in\{\exp(a_{(0)})\mid a\in\h\}$ by \eqref{Eq:idh}.
By Lemma \ref{L:innM} (1), $\sigma_{|V_L}$ acts trivially on  $\irr(V_L)$, and hence $\sigma(V_{\lambda+L}\otimes \varphi({\lambda+L}))=V_{\lambda+L}\otimes \varphi({\lambda+L})$ for all $\lambda+L\in\mathcal{D}(L)$.
Since $\varphi$ is a bijection from $\mathcal{D}(L)$ to $\irr(W)$, $\sigma_{|W}\in\Aut(W)$ also acts trivially on $\irr(W)$.
By Theorem \ref{T:inj}, the action of $\Aut(W)$ on $\irr(W)$ is faithful.
Hence we have $\sigma_{|W}=id$.
It follows from \eqref{Eq:ns0} that $\sigma\in S_\varphi^*\{\exp(a_{(0)})\mid a\in\h\}$.
By \eqref{Eq:S*}, we have $\sigma=\exp(u_{(0)})$ for some $u\in\h$.
\end{proof}

\begin{remark} If $V$ is isomorphic to a Niemeier lattice VOA, then $K(V)\subset \Inn(V)$ by Remark \ref{R:KOknown}.
Hence for any holomorphic VOA $V$ of central charge $24$ with $V_1\neq0$, we have $K(V)\subset\Inn(V)$, which proves the first assertion of Theorem \ref{T:main11}.
\end{remark}

Let $V_1=\g=\bigoplus_{i=1}^s\g_i$, where $\g_i$ are simple ideals, and let $k_i$ be the level of $\g_i$.
Note that $k_i\in\Z_{>0}$ (\cite{DM06}).
The norm of roots in $\g$ is normalized so that $\langle\alpha|\alpha\rangle=2$ for any long roots $\alpha$.

\begin{proposition}\label{L:KV2} Let $Q^i$ be the root lattice of $\g_i$ and set $\tilde{Q}=\bigoplus_{i=1}^s\frac{1}{\sqrt{k_i}}Q^i$.
Then $$K(V)=\{\exp(2\pi\sqrt{-1}v_{(0)})\mid v+L\in \tilde{Q}^*/L\}$$ and it is isomorphic to $L^*/\tilde{Q}$ as a group.
\end{proposition}
\begin{proof}
By \eqref{Eq:V} and $\h\subset (V_L)_1$, for $x\in\h$, $\exp(2\pi\sqrt{-1}x_{(0)})=id$ on $V$ if and only if $x\in L$.
By Proposition \ref{L:idh}, we have $K(V)\subset\{\exp(2\pi\sqrt{-1}v_{(0)})\mid v+L\in\h/L\}$.

Recall from \cite{DM06} that the subVOA generated by $V_1$ is isomorphic to $\bigotimes_{i=1}^sL_{\widehat{\g_i}}(k_i,0)$, where $L_{\widehat{\g_i}}(k_i,0)$ is the simple affine VOA associated with $\g_i$ at level $k_i$.
It was proved in \cite{DR} that $L_{\widehat{\mathfrak{g}_{i}}}(k_i,0)$ is a simple current extension of $V_{\sqrt{k_i}Q^i_{long}}\otimes K(\g_i,k_i)$ as follows:
\begin{align} L_{\widehat{\mathfrak{g}_{i}}}(k_i,0)\cong \bigoplus_{\lambda\in (1/\sqrt{k_i})Q^i/\sqrt{k_i}Q^i_{long}}V_{\lambda+\sqrt{k_i}Q^i_{long}}\otimes M^{0,\lambda},\label{Eq:Lg1}
\end{align}
where $Q^i_{long}$ is the sublattice of the root lattice $Q^i$ spanned by long roots, $K(\g_i,k_i)$ is the parafermion VOA and $M^{0,\lambda}$ are certain irreducible $K(\mathfrak{g}_i,k_i)$-modules.

By \eqref{Eq:Lg1}, for $v\in\h$, $\exp(2\pi\sqrt{-1}v_{(0)})=id$ on $V_1$ if and only if 
$v\in \tilde{Q}^*$.
Hence $K(V)=\{\exp(2\pi\sqrt{-1}v_{(0)})\mid v+L\in \tilde{Q}^*/L\}$.
Clearly, this group is isomorphic to the dual $L^*/\tilde{Q}$ of $\tilde{Q}^*/L$.
\end{proof}

\begin{remark}\label{R:shortQi} For a short root $\beta$ in the root lattice $Q^i$ of $\g_i$, we have $\langle \beta|\beta\rangle=2/r_i$, where $r_i$ is the lacing number of $\g_i$.
Hence $Q^i$ is not necessarily even.
\end{remark}

Later, we use the sublattice
\begin{equation}
Q_\g=\bigoplus_{i=1}^s\sqrt{k_i}Q^i_{long}\subset L.\label{Eq:LQl}
\end{equation}
Note that the ranks of both $Q_\g$ and $L$ are equal to $\dim\h$.

\subsection{The group $\Out(V)$}
Let $V$ be a holomorphic VOA of central charge $24$ with $0<\rank V_1<24$.
Let $\h$ be a Cartan subalgebra of $V_1$.
Set $W=\Com_V(\h)$ and $V_L=\Com_V(W)$ as in Section \ref{S:31}.
In this subsection, we describe $\Out(V)$, defined in the introduction, in terms of $V_1$ and $L$.

As discussed in the previous section, $V$ is a simple current extension $V=\bigoplus_{M\in S_\varphi}M$.
Hence the fixed-point subVOA of $S_\varphi^*$ is 
\begin{equation}
V^{S_\varphi^*}=\{v\in V\mid \sigma v=v\quad \text{for all }\sigma\in S_\varphi^*\}=V_L\otimes W.\label{Eq:VS*}
\end{equation}
It follows that the normalizer of $S_\varphi^*$ in $\Aut(V)$ is given by
\begin{align}
N_{\Aut(V)}(S_\varphi^*)=\{\sigma\in\Aut(V)\mid \sigma(V_{L}\otimes W)=V_{L}\otimes W\}.\label{Eq:NS*}
\end{align}
By \cite[Theorem 3.3]{Sh04}, we obtain
\begin{align}
N_{\Aut(V)}(S_\varphi^*)/S_\varphi^*&\cong \Stab_{\Aut(V_{L}\otimes W)}(S_\varphi)=\{\sigma\in \Aut(V_{L}\otimes W)\mid S_\varphi\circ \sigma=S_\varphi\}.\label{Eq:ns}
\end{align}
Recall that $\h$ is the fixed Cartan subalgebra of $V_1$.
Set
$${\rm Stab}_{\Aut(V)}(\h)=\{\sigma\in\Aut(V)\mid \sigma(\h)=\h\},\quad {\rm Stab}_{\Inn(V)}(\h)={\rm Stab}_{\aut(V)}(\h)\cap\Inn(V).$$
\begin{lemma}\label{L:CSA}
\begin{enumerate}[{\rm (1)}]
\item $\aut(V) = \Inn(V){\rm Stab}_{\Aut(V)}(\h)$;
\item $\Out(V)\cong {\rm Stab}_{\Aut(V)}(\h)/{\rm Stab}_{\Inn(V)}(\h)$;
\item $N_{\Aut(V)}(S^*_\varphi)=\Inn(V_L){\rm Stab}_{\Aut(V)}(\h)$.
\end{enumerate}
\end{lemma}
\begin{proof}
 
Let $\sigma\in \aut(V)$. 
Since all Cartan subalgebras of $V_1$ are conjugate under $\Inn(V_1)$, there exists $\tau\in\Inn(V)$ such that $\tau\sigma(\mathfrak{h})= \mathfrak{h}$.
Hence $\tau\sigma\in{\rm Stab}_{\aut(V)}(\h)$, which proves (1). 
Clearly, the assertion (1), Lemma \ref{L:idh} and the definition of $\Out(V)$ imply (2).

It follows from $\Com_V(\h)=W$ and $\Com_V(W)=V_L$ that ${\rm Stab}_{\aut(V)}(\h)$ preserves $V_L\otimes W$.
Hence by \eqref{Eq:NS*}, ${\rm Stab}_{\aut(V)}(\h)\subset N_{\Aut(V)}(S_\varphi^*)$.
In addition, by Lemma \ref{L:innM}, $\Inn(V_L)$ preserves $S_\varphi$.
Hence by \eqref{Eq:ns}, we have $\Inn(V_L)\subset N_{\Aut(V)}(S_\varphi^*)$.
Thus $\Inn(V_L){\rm Stab}_{\Aut(V)}(\h)\subset N_{\Aut(V)}(S^*_\varphi)$.

Let $\sigma\in N_{\Aut(V)}(S_\varphi^*)$.
By \eqref{Eq:NS*}, $\sigma$ preserves $V_L\otimes W$, and therefore also $(V_L\otimes W)_1=(V_L)_1\otimes\1$.
Since $\h$ is a Cartan subalgebra of $(V_L)_1$, there exists $\tau\in \Inn(V_L)$ such that $\tau\sigma(\h)=\h$.
Hence $\sigma\in \Inn(V_L){\rm Stab}_{\Aut(V)}(\h)$.
\end{proof}

\begin{lemma}\label{Lem:Stabh} $\Stab_{\Aut(V)}(\h)/S_\varphi^*\cong \Stab_{\Aut(V_{L}\otimes W)}(S_\varphi)\cap\Stab_{\Aut(V_{L}\otimes W)}(\h)$.
\end{lemma}
\begin{proof} By Lemma \ref{L:CSA} (3), $\Stab_{\Aut(V)}(\h)\subset N_{\Aut(V)}(S_\varphi^*)$.
By \eqref{Eq:ns}, $\Stab_{\Aut(V)}(\h)/S_\varphi^*\subset \Stab_{\Aut(V_{L}\otimes W)}(S_\varphi)$.
Hence $\Stab_{\Aut(V)}(\h)/S_\varphi^*\subset \Stab_{\Aut(V_{L}\otimes W)}(S_\varphi)\cap\Stab_{\Aut(V_{L}\otimes W)}(\h)$.
For $\sigma\in \Stab_{\Aut(V_{L}\otimes W)}(S_\varphi)\cap\Stab_{\Aut(V_{L}\otimes W)}(\h)$, by \eqref{Eq:ns0} and \eqref{Eq:ns}, there exists $\tilde{\sigma}\in N_{\Aut(V)}(S_\varphi^*)$ such that $\tilde{\sigma}_{|V_L\otimes W}=\sigma$ and $\tilde{\sigma}(\h)=\h$.
Hence $\sigma\in\Stab_{\Aut(V)}(\h)/S_\varphi^*$.
\end{proof}
It follows from $(V_L\otimes W)_1=(V_L)_1\otimes\1$
and $\Com_{V_{L}\otimes W}((V_{L})_1\otimes \1)=\1\otimes W$ that  
\begin{equation*}
\Aut(V_{L}\otimes W)\cong\Aut(V_{L})\times \Aut(W).\label{Eq:aut2}
\end{equation*}
Hence we obtain the group homomorphism \begin{equation}
\Aut(V_L\otimes W)\to  O(\irr(V_L),q_{V_L})\times O(\irr(W),-q_W),\quad \sigma\mapsto (\mu_{V_L}(\sigma_{|V_L}),{\mu}_W(\sigma_{|W})).\label{Eq:AutLW}
\end{equation}
Here we view $\mu_W(\sigma_{|W})\in O(\irr(W),-q_W)$ via $O(\irr(W),q_W)= O(\irr(W),-q_W)$.
By the injectivity of $\mu_W$ (Theorem \ref{T:inj}),
$$\Aut_0(W)=1,\qquad \overline{\Aut}(W)\cong\Aut(W);$$
we often identify $\overline{\Aut}(W)$ with $\Aut(W)$.
Hence the kernel of the homomorphism \eqref{Eq:AutLW} is $\Aut_0({V_L})\times 1$.
By \eqref{Eq:S} and \eqref{Eq:ns}, we have
\begin{align*}
\Stab_{\Aut(V_{L}\otimes W)}(S_\varphi)&\cong (\Aut_0({V_L})\times 1).\{(k,\varphi k\varphi^{-1})\mid k\in\overline{\Aut}(V_L),\ \varphi k\varphi^{-1}\in\overline{\Aut}(W)\}.
\end{align*}
We now identify $(\irr(V_L),q_{V_L})$ with $(\mathcal{D}(L),q_L)$.
Note that $\overline{\Aut}({V_L})\cong\overline{O}({L})$ (see Lemma \ref{L:XL}).
Considering the restriction of $\Stab_{\Aut(V_{L}\otimes W)}(S_\varphi)$ to $V_L$, we have 
\begin{equation}
\Stab_{\Aut(V_{L}\otimes W)}(S_\varphi)\cong \Aut_0({V_L}).(\overline{O}(L)\cap \varphi^{*}(\overline{\Aut}(W))),\label{Eq:StabS}
\end{equation}
where 
$$\varphi^*(\overline{\Aut}(W))=\varphi^{-1}(\overline{\Aut}(W))\varphi\subset O(\mathcal{D}(L),q_L).$$
By Lemma \ref{L:XL}, \eqref{Eq:stabh} and \eqref{Eq:StabS}, we have
\begin{align*}
&\Stab_{\Aut(V_{L}\otimes W)}(S_\varphi)\cap\Stab_{\Aut(V_{L}\otimes W)}(\h)\\\cong& \{\exp(a_{(0)})\mid a\in \h\}\iota^{-1}(O_0(L)).(\overline{O}({L})\cap \varphi^{*}(\overline{\Aut}(W))).
\end{align*}
By \eqref{Eq:exact} and \eqref{Eq:stabhVL}, 
\begin{align}
\begin{split}
\left(\Stab_{\Aut(V_{L}\otimes W)}(S_\varphi)\cap\Stab_{\Aut(V_{L}\otimes W)}(\h)\right)_{|\h}&
\cong O_0(L).(\overline{O}({L})\cap \varphi^{*}(\overline{\Aut}(W)))
\\
&\cong\mu_L^{-1}(\overline{O}({L})\cap \varphi^{*}(\overline{\Aut}(W))).\end{split}\label{Eq:StabS2}
\end{align}

Let $W(V_1)$ denote the Weyl group of the semisimple Lie algebra $V_1$.

\begin{lemma}\label{L:Wg}
\begin{enumerate}[{\rm (1)}]
\item ${\rm Stab}_{\Inn(V)}(\h)/\{\exp(a_{(0)})\mid a\in \h\}\cong {W}(V_1)$.
\item  ${\rm Stab}_{\Aut(V)}(\h)/\{\exp(a_{(0)})\mid a\in \h\}\cong \mu_L^{-1}(\overline{O}({L})\cap \varphi^{*}(\overline{\Aut}(W)))$.
\end{enumerate}
\end{lemma}
\begin{proof}
Since $V_1$ is a semisimple Lie algebra, ${\rm Stab}_{\Inn(V)}(\h)$ acts on $\h$ as $W(V_1)$. 
Hence (1) follows from Proposition \ref{L:idh}.
Combining Proposition \ref{L:idh} and \eqref{Eq:StabS2}, we obtain (2).
\end{proof}

By Lemmas \ref{L:CSA} (2) and \ref{L:Wg}, we obtain the following:

\begin{proposition}\label{P:OV} 
$\Out(V)\cong \mu_L^{-1}(\overline{O}({L})\cap \varphi^{*}(\overline{\Aut}(W)))/W(V_1).$
\end{proposition}

As a corollary, we obtain
\begin{align}
\begin{split}
|O(L)/W(V_1):\Out(V)|&=|O(L):\mu_L^{-1}(\overline{O}({L})\cap \varphi^{*}(\overline{\Aut}(W)))|\\
&=|\overline{O}(L):(\overline{O}({L})\cap \varphi^{*}(\overline{\Aut}(W)))|.
\end{split}\label{Eq:OV}
\end{align}

Moreover, we obtain the following:

\begin{lemma}\label{L:Out} 
Assume that the conjugacy class of $g\in O(\Lambda)$ is neither $2C$, $6G$ nor $10F$ and that $\overline{O}(L)$ and $\varphi^{*}(\overline{\Aut}(W))$ generate $O(\mathcal{D}(L),q_L)$.
Then 
$$|O(L)/W(V_1):\Out(V)|=|O(\irr(W),q_W):\overline{\Aut}(W)|.$$
In particular, 
\begin{enumerate}[{\rm (1)}]
\item if the conjugacy class of $g$ is $2A$ or $6E$, then $\Out(V)\cong O(L)/W(V_1)$;
\item if the conjugacy class of $g$ is $3B$, $4C$, $7B$ or $8E$, then $\Out(V)\cong O(L)/\langle W(V_1),-1\rangle$.
\end{enumerate}
\end{lemma}
\begin{proof} By Lemma \ref{Lem:center} (1), $\overline{\Aut}(W)$ is normal in $O(\irr(W),q_W)$.
The equation \eqref{Eq:OV} and the group isomorphism theorem show 
\begin{align*}
|O(L)/W(V_1):\Out(V)|
=|O(\mathcal{D}(L),q_L):\varphi^{*}(\overline{\Aut}(W))|=|O(\irr(W),q_W):\overline{\Aut}(W)|.
\end{align*}
The assertion (1) follows from $O(\irr(W),q_W)=\overline{\Aut}(W)$ in Table \ref{Table:main}.
The assertion (2) follows from Lemma \ref{Lem:center} (2), Table \ref{Table:main} and the fact that the $-1$-isometry in $O(L)$ gives the $-1$-isometry in $\overline{O}(L)$.
\end{proof}

\subsection{Weight one Lie algebra structures and orbit lattices}
Let $V$ be a strongly regular holomorphic VOA of central charge $24$ with $0<\rank V_1<24$.
Then $\g=V_1$ is semisimple. Let $\g=\bigoplus_{i=1}^s\g_{i,k_i}$, where $\g_{i,k_i}$ is a simple ideal with level $k_i\in\Z_{>0}$.

\begin{remark} The level $k_i$ of a simple ideal $\g_i$ of $V_1$ is determined by the 
following formula in \cite{Sc93,DM04a}:
\begin{equation}
\frac{h^\vee_i}{k_i}=\frac{\dim V_1-24}{24},\label{Eq:DM04a}
\end{equation}
where $h^\vee_i$ is the dual Coxeter number of $\g_i$.
\end{remark}

Let $\mathfrak{h}$ be a Cartan subalgebra of $V_1$.
By Theorem \ref{T:Ho}, $W=\Com_V(\h)\cong V_{\Lambda_g}^{\hat{g}}$ for some 
$g\in O(\Lambda)$ belonging to the $10$ conjugacy classes.
In addition, $\Com_V(W)\cong V_{L}$ for some even lattice $L$.
In this subsection, we describe some properties of $L$ by using $\g=V_1$.

Set $\ell=2|g|$ if $g\in 2C,6G,10F$ and $\ell=|g|$ otherwise (cf.\ Remark \ref{R:ell}).
By Lemma \ref{L:level}, $L$ has level $\ell$, and $\sqrt{\ell}L^*$ is an even lattice.

\begin{proposition}\label{P:Ug} 
Let $Q^i$ be the root lattice of $\g_i$; 
here the norm of roots in $\g_i$ is normalized so that $\langle \alpha|\alpha\rangle=2$ for any long roots $\alpha$ (cf. Remark \ref{R:shortQi}).
Then the even lattice $U=\sqrt{\ell}L^*$ contains 
\begin{equation}
P_\g=\bigoplus_{i=1}^s\frac{\sqrt{\ell}}{\sqrt{k_i}}Q^i\label{Eq:Ug}
\end{equation}
and $\rank U=\rank P_\g$.
Moreover, 
if the vector $v=\dfrac{\sqrt{\ell}}{\sqrt{k_i}}\beta$ associated with a root $\beta$ of $\g_i$ is primitive in $U$, then $v$ is a root of $U$.
\end{proposition}

\begin{proof} Recall that the ratio of the normalized killing form on $\g_i$ and the bilinear form $\langle\ |\ \rangle$ on $L\subset\h$ is $k_i$.
Hence $\bigoplus_{i=1}^s(1/\sqrt{k_i})Q^i$ 	is the set of weights for $\h$ of the subVOA generated by $V_1$ with respect to the bilinear form $\langle\ |\ \rangle$ (see also \eqref{Eq:Lg1}).  
By \eqref{Eq:V}, we have $(1/\sqrt{k_i})Q^i\subset L^*$, which shows the former assertion (cf. the proof of Proposition \ref{L:KV2}).

Set $r_\beta=1$ (resp. $r_\beta=r_i)$ if $\beta$ is long (resp. short), where $r_i$ is the lacing number of $\g_i$.
Then $r_\beta\beta$ belongs to the even lattice $Q^i_{long}$ generated by long roots of $Q^i$, and $\sqrt{k_i}r_\beta\beta\in \sqrt{k_i}Q^i_{long}\subset L$ (see \eqref{Eq:LQl}).
In addition, $\langle v|v\rangle/2=\ell/(k_ir_\beta)$.
Hence $$v=\frac{\ell}{k_ir_\beta}\frac{1}{\sqrt{\ell}}\sqrt{k_i}r_\beta\beta\in\frac{\langle v|v\rangle}{2}\frac{1}{\sqrt{\ell}}L=\frac{\langle v|v\rangle}{2}U^*.$$
Thus the reflection $\sigma_v$ preserves $U$, and $v$ is a root of $U$.
\end{proof}

\begin{remark}\label{R:Pg} The lattice $P_\g$ is equal to $\sqrt{\ell}\tilde{Q}$, where $\tilde{Q}$ is defined in Proposition \ref{L:KV2}.
\end{remark}

By the classification of irreducible $W$-modules (cf.\ \cite{Lam20}), we obtain the following lemma:

\begin{lemma}\label{Lem:cwtW} {\rm (cf.\ \cite{Lam20})} Assume that the conjugacy class of $g$ is $2A,3B,5B$ or $7B$.
Let $M$ be an irreducible $W$-module.
If $M$ is not isomorphic to $W$, then $\rho(M)\ge (\ell-1)/\ell$ and $\rho(M)\in(1/\ell)\Z$.
Moreover, if $\rho(M)\in (\ell-1)/\ell+\Z$, then $\rho(M)=(\ell-1)/\ell$.
\end{lemma}

\begin{proposition}\label{P:root} 
Assume that 
the conjugacy class of $g$ is $2A$, $3B$, $5B$ or $7B$.
Then $U=\sqrt{\ell}L^*$ is a level $\ell$ lattice.
Moreover, the root system of $U$ and the root system of the semisimple Lie algebra $V_1=\g$ 
have the same type.
In particular, the sublattice $P_\g$ in \eqref{Eq:Ug} of $U$ is generated by roots of $U$.
\end{proposition}

\begin{proof}
By Lemma \ref{L:level}, $U$ has level $\ell$.
Since $W_1=0$ and $\ell$ is prime,
by \eqref{Eq:V} and Lemma \ref{Lem:cwtW}, we have
\begin{equation}
V_1=(V_{L})_1\otimes \1\oplus \bigoplus_{\min(\lambda+L)=2/\ell} (V_{\lambda+L})_{1/\ell}\otimes \varphi({\lambda+L})_{1-1/\ell},\label{Eq:wt1V}
\end{equation}
where $\min(\lambda+L)=\min\{\langle x|x\rangle\mid x\in \lambda+L\}$.
Then the roots of $V_1$ with respect to $\h$ are given by $$\{\alpha\in L\mid \langle\alpha|\alpha\rangle=2\}\cup\{\alpha\in L^*\mid \langle\alpha|\alpha\rangle=2/\ell\}.$$
We can rewrite it as follows:
$$\{\alpha\in \ell U^*\subset U\mid \langle\alpha|\alpha\rangle=2\ell\}\cup\{\alpha\in U\mid \langle \alpha|\alpha\rangle=2\}.$$
By Lemma \ref{RL}, this set is $R(U)$.
Since $\ell$ is prime, for any root $\beta\in Q^i$, the vector $(\sqrt{\ell}/\sqrt{k_i})\beta$ is primitive in 
$U=\sqrt{\ell}L^*$.
Hence by Proposition \ref{P:Ug}, the root systems of $V_1$ and $R(U)$ 
have the same type.
The last assertion also follows from Proposition \ref{P:Ug}.
\end{proof}

\subsection{Schellekens' list and isometries of the Leech lattice}\label{S:Sch}
Let $V$ be a holomorphic VOA of central charge $24$ such that $0<\rank V_1<24$.  
Set $\g=V_1$.
Then $\g$ is one of $46$ semisimple Lie algebras in Schellekens' list (\cite{Sc93}).
Let $\h$ be a Cartan subalgebra of $\g$.
By Theorem \ref{T:Ho}, $W=\Com_V(\h)\cong V_{\Lambda_g}^{\hat{g}}$ for some 
$g\in O(\Lambda)$ belonging to the $10$ conjugacy classes.
In addition, the conjugacy class of $g$ is uniquely determined by $\g$, which is summarized in Table \ref{T:Lie} (see also \cite[Tables 6--15]{Ho2} and \cite[Table 2]{ELMS}).
Here the symbol $X_{n,k}$ denotes (the type of) a simple Lie algebra whose type is $X_n$ and level is $k$.

{\small
\begin{longtable}{|c|c|c|c|c|}
\caption{Weight one Lie algebras of holomorphic VOAs of central charge $24$ associated with $V_{\Lambda_g}^{\hat{g}}$ for $g\in O(\Lambda)$}\label{T:Lie}
\\ \hline 
Genus&Class& $\#$ of $L$& $\#$ of $V_1$& Weight one Lie algebra structures\\ \hline
$B$&$2A$& $17$&$17$&$A_{1,2}^{16}$, $A_{3,2}^4A_{1,1}^4$, $D_{4,2}^2B_{2,1}^4$, $A_{5,2}^2C_{2,1}A_{2,1}^2$, $D_{5,2}^2C_{2,1}A_{2,1}^2$, $A_{7,2}C_{3,1}^2A_{3,1}$,\\
&&&& $C_{4,1}^4$, $D_{6,2}C_{4,1}B_{3,1}^2$, $A_{9,2}A_{4,1}B_{3,1}$, $E_{6,2}C_{5,1}A_{5,1}$, $D_{8,2}B_{4,1}^2$, $C_{6,1}^2B_{4,1}$,\\
&&&& $D_{9,2}A_{7,1}$, $C_{8,1}F_{4,1}^2$, $E_{7,2}B_{5,1}F_{4,1}$, $C_{10,1}B_{6,1}$, $B_{8,1}E_{8,2}$\\ \hline
$C$&$3B$&$6$&$6$& $A_{2,3}^6$, $A_{5,3}D_{4,3}A_{1,1}^3$, $A_{8,3}A_{2,1}^2$, $E_{6,3}G_{2,1}^3$, $D_{7,3}A_{3,1}G_{2,1}$, $E_{7,3}A_{5,1}$\\
$D$&$2C$&$2$&$9$&$A_{1,4}^{12}$, $B_{2,2}^6$, $B_{3,2}^4$, $B_{4,2}^3$, $B_{6,2}^2$, $B_{12,2}$, $D_{4,4}A_{2,2}^4$, $C_{4,2}A_{4,2}^2$, $A_{8,2}F_{4,2}$\\
$E$&$4C$&$5$&$5$& $A_{3,4}^3A_{1,2}$, $D_{5,4}C_{3,2}A_{1,1}^2$, $A_{7,4}A_{1,1}^3$, $E_{6,4}A_{2,1}B_{2,1}$, $C_{7,2}A_{3,1}$ \\
$F$&$5B$&$2$&$2$& $A_{4,5}^2$, $D_{6,5}A_{1,1}^2$\\
$G$&$6E$&$2$&$2$& $A_{5,6}B_{2,3}A_{1,2}$, $C_{5,3}G_{2,2}A_{1,1}$\\ 
$H$&$7B$&$1$&$1$& $A_{6,7}$\\
$I$&$8E$&$1$&$1$& $D_{5,8}A_{1,2}$\\
$J$&$6G$&$1$&$2$& $D_{4,12}A_{2,6}$, $F_{4,6}A_{2,2}$\\
$K$&$10F$&$1$&$1$& $C_{4,10}$\\
\hline 
\end{longtable}
}

\begin{remark}
It would be possible to classify orbit lattices by the rank and the quadratic space structure on the discriminant group; in fact, the number of isometric classes of orbit lattices is given in \cite[Table 4]{Ho2} (see Table \ref{T:Lie}). 
We will explicitly describe the orbit lattice $L_\g$ corresponding to $\g$ in Section \ref{S:5}.
Note that the orbit lattices have been described in \cite{Ho2} by using Niemeier lattices.
\end{remark}

\begin{remark} 
By Table \ref{T:Lie}, we observe 
$$\ell=\mathrm{lcm}(\{r_1k_1,r_2k_2,\dots,r_sk_s\}),$$ where $r_i$ is the lacing number of $\g_i$ and $\ell=2|g|$ if $g\in 2C,6G,10F$ and $\ell=|g|$ otherwise (See also  Remark \ref{R:ell}).
\end{remark}

\section{Inequivalent simple current extensions} \label{sec:4}
Let $W$ be one of the $10$ VOAs in Theorem \ref{T:Ho} and let $L$ be an even lattice satisfying \eqref{Eq:quad0} and \eqref{Eq:quad}.
In this subsection, we determine the number of holomorphic VOAs of central charge $24$ obtained as inequivalent simple current extensions of $V_L\otimes W$ based on the arguments in \cite{Ho2}.

Let $\mathcal{O}$ be the set of all isometries from $(\mathcal{D}(L),q_{L})$ to $(\irr(W),-q_W)$.
For $\psi\in \mathcal{O}$, 
$$V_\psi=\bigoplus_{\lambda+L\in\mathcal{D}(L)}V_{\lambda+L}\otimes\psi(\lambda+L)$$
has a holomorphic VOA structure of central charge $24$ as a simple current extension of $V_L\otimes W$ (\cite[Theorem 4.2]{EMS}).
Define $S_\psi=\{(V_{\lambda+L},\psi(\lambda+L))\mid \lambda+L\in\mathcal{D}(L)\}$ as in \eqref{Eq:S}.

Let $f\in \overline{\Aut}({W})$, $h\in \overline{O}({L})$ and $\psi\in\mathcal{O}$. Then  $f\circ \psi\circ h$ also belongs to $\mathcal{O}$ and $S_{f\circ \psi\circ h}\circ (h,f^{-1})=S_\psi$.
Hence $(h,f^{-1})$ induces an isomorphism between the holomorphic VOAs $V_\psi$ and  $V_{f\circ \psi\circ h}$. 
Conversely, we assume that $\psi,\psi'\in\mathcal{O}$ satisfy $V_\psi\cong V_{\psi'}$ as simple current extensions of $V_L\otimes W$, that is, there exists an isomorphism $\xi:V_\psi\to V_{\psi'}$ such that $\xi(V_L\otimes W)=V_L\otimes W$.
Then $S_\psi$ and $S_{\psi'}$ are conjugate by the restriction of $\xi$ to $V_L\otimes W$.
Note that $\Aut(V_L\otimes W)\cong\Aut(V_L)\times\Aut(W)$ and $\overline{\Aut}(V_L)$ is identified with $\overline{O}(L)$ (see Lemma \ref{L:XL}).
Therefore, the number of holomorphic VOAs obtained by inequivalent simple current extensions $\{V_\varphi\mid \varphi\in\mathcal{O}\}$ of $V_L\otimes W$ is equal to the number of double cosets in 
$$ \overline{\Aut}(W)\backslash \mathcal{O}/\overline{O}({L}).$$

\begin{remark}
In general, inequivalent simple current extensions may become isomorphic VOAs.
Fortunately, in our cases, this does not happen;  
see Propositions \ref{P:uni2-8}, \ref{P:uni5B}, \ref{P:uni2Ca}, \ref{P:uni2Cb}, \ref{P:uni6G} and \ref{P:uni10F}.
\end{remark}

Now fix an isometry $i\in \mathcal{O}$. Then $i^*(h)= i^{-1}\circ h\circ i\in O(\mathcal{D}(L),q_{L})$ for any $h\in O(\irr(W),-q_W)$. 
We consider the double cosets in  $i^*(\overline{\Aut}(W))\backslash  O(\mathcal{D}(L),q_{L}) / \overline{O}({L})$. 
Note that $i\circ f\in \mathcal{O}$ for any $f\in  O(\mathcal{D}(L),q_{L})$. 
Conversely, $i^{-1}\circ \psi\in O(\mathcal{D}(L),q_{L})$ for any $\psi\in \mathcal{O}$.  
Therefore, $i$ induces a bijective map between $\mathcal{O}$ and $O(\mathcal{D}(L),q_{L})$,  which gives the following:  

\begin{proposition}\label{P:Hoext}{\rm \cite[Theorem 2.7]{Ho2}}
Let $\psi,\psi'\in \mathcal{O}$. Then $\psi$ and $\psi'$ are in the same double coset of $ \overline{\Aut}(W) \backslash \mathcal{O}/\overline{O}({L})$ if and only if $ i^{-1}\circ\psi$ and $i^{-1}\circ\psi'$ are in the same double coset of  
$i^*(\overline{\Aut}(W))\backslash  O(\mathcal{D}(L),q_{L}) / \overline{O}({L})$.
In particular, the number of inequivalent simple current extensions in $\{V_\varphi\mid \varphi\in\mathcal{O}\}$ is equal to $|i^*(\overline{\Aut}(W))\backslash  O(\mathcal{D}(L),q_{L}) / \overline{O}({L})|$.
\end{proposition}

The following proposition proves the conjecture \cite[Conjecture 4.8]{Ho2} for six conjugacy classes.
The other four cases will be discussed in Section \ref{S:5}.
Some cases were discussed in \cite[Remark 4.9]{Ho2}.

\begin{proposition}\label{P:uni2-8} Let $g\in O(\Lambda)$ such that $W\cong V_{\Lambda_g}^{\hat{g}}$.
Assume that the conjugacy class of $g$ is $2A$, $3B$, $4C$, $6E$, $7B$ or $8E$.
Then, for each $L$ satisfying \eqref{Eq:quad0} and \eqref{Eq:quad}, there exists exactly one holomorphic VOA of central charge $24$ obtained as a simple current extension of $V_L\otimes W$, up to isomorphism.
\end{proposition}
\begin{proof} By Proposition \ref{P:Hoext}, it suffices to show that $|i^*(\overline{\Aut}(W))\backslash  O(\mathcal{D}(L),q_{L}) / \overline{O}({L})|=1$, that is, $i^*(\overline{\Aut}(W))\overline{O}({L})=O(\mathcal{D}(L),q_{L})$.

If the conjugacy class of $g$ is $2A$ or $6E$, then the assertion is obvious since $i^*(\overline{\Aut(W)})\cong\Aut(W) \cong O(\mathcal{D}(L),q_{L})$ by Table \ref{Table:main}.

If the conjugacy class of $g$ is $3B$, $4C$, $7B$ or $8E$, then $|O(\mathcal{D}(L),q_{L}):i^*(\overline{\Aut}(W))|=2$ by Table \ref{Table:main}.
In addition, the $-1$-isometry of $\mathcal{D}(L)$ belongs to $\overline{O}(L)$  but it does not belong to $i^*(\overline{\Aut}(W))$ by Proposition \ref{Lem:center} (2).
Hence we obtain the desired result.
\end{proof}

The following lemma, which will be used to determine the number of double cosets, is probably well-known.
 
\begin{lemma}\label{DCvsConj}
Let $G$ be a finite group and let $G_1, G_2$ be subgroups of $G$.  
Suppose $N_{G}(G_2)=G_2$. Then $a, a'$ are in the same double coset of $G_2\backslash G/G_1$ if and only if $b^{-1}a^{-1}G_2ab=a'^{-1}G_2a'$ for some $b\in G_1$.
\end{lemma}

\begin{proof}
Suppose $a'\in G_2aG_1$. Then $a'=a_2 aa_1$ for some $a_1\in G_1$ and $a_2\in G_2$. 
Then $a'^{-1} G_2a'= a_1^{-1}a^{-1}a_2^{-1}G_2a_2 aa_1 =a_1^{-1}(a^{-1}G_2 a)a_1$.

Conversely, we suppose 
$a'^{-1}G_2a' =a_1^{-1}a^{-1}G_2a a_1$ for some $a_1\in G_1$.
Then $
G_2 =a'a_1^{-1} a^{-1}G_2a a_1a'^{-1}$.  Since $N_{G}(G_2)=G_2$, we have $a a_1a'^{-1} =a_2\in G_2$ and $a' = a_2^{-1}aa_1$ as desired. 
\end{proof}

\begin{remark}\label{DCcount}
Under the same assumptions as in Lemma \ref{DCvsConj}, the number of double cosets of $G_2\backslash G/G_1$ is equal to the number of $G_1$-orbits on the set $\{a^{-1} G_2 a\mid  a\in G\}$ of all subgroups of $G$ conjugate to $G_2$ by conjugation.
\end{remark}

\section{Automorphism groups of holomorphic VOAs of central charge $24$}\label{S:5}
Let $V$ be a (strongly regular) holomorphic VOA of central charge $24$ with $0<\rank V_1<24$.
Fix a Cartan subalgebra $\h$ of $V_1$.
By Theorem \ref{T:Ho}, $W=\Com_V(\h)\cong V_{\Lambda_g}^{\hat{g}}$ for some $g\in O(\Lambda)$ belonging to the $10$ conjugacy classes. 
Note that $\Com_V(W)$ is a lattice VOA $V_L$ and 
the conjugacy class of $g$ is uniquely determined by the Lie algebra structure of $V_1$ (see Table \ref{T:Lie}). 
In addition, $V$ is a simple current extension of $V_L\otimes W$.

In this section, by using the Lie algebra structure of $\g=V_1$ in Schellekens' list, we describe the orbit lattice $L$ explicitly, which implies that $L=L_\g$ is uniquely determined by $\g$, up to isometry.
For each $\g$, we also determine the group structures of $K(V)$ and $\Out(V)$ based on the case-by-case analysis on $W$ and $L_\g$.
For the conjugacy classes of $g$ that we have not dealt with in Proposition \ref{P:uni2-8}, we also determine the number of holomorphic VOAs obtained as inequivalent simple current extensions of $V_L\otimes W$.

\begin{remark}
Based on a similar method, some partial results for the conjugacy classes $2A,3B,5B$ and $7B$ and $2C$ were obtained in \cite{LS17} and \cite{HS}, respectively.
\end{remark}

\begin{remark}
In the tables of this section, $\Sym_n$, $\Alt_n$ and $\Dih_n$ denote the symmetric group of degree $n$, the alternating group of degree $n$ and the dihedral group of order $n$, respectively.
\end{remark}

\subsection{Conjugacy class $2A$ (Genus $B$)}

Assume that $g$ belongs to the conjugacy class $2A$ of $O(\Lambda)$.
Then $O(\irr(W),q_W)\cong GO^+_{10}(2)\cong \Omega_{10}^+(2).2$.
By Table \ref{Table:main}, $\Aut(W)(\cong\overline{\Aut}(W))$ has the shape $GO_{10}^+(2)$, which is the full orthogonal group $O(\irr(W),q_W)$.

Since the central charge of $W$ is $8$, $L$ is an even lattice of rank $16$ such that $(\mathcal{D}(L),q_L)\cong(\irr(W),-q_W)$.
Then $\mathcal{D}(L)\cong \Z_2^{10}$.
Set $U=\sqrt{2}L^*$.
Then $\mathcal{D}(U)\cong \Z_2^6$, and 
by Proposition \ref{P:root}, 
 $U$ is a level $2$ lattice.
Such lattices $U$ were classified in \cite{RV}.
Furthermore that can now  be verified easily using MAGMA.
More precisely, it was proved in \cite[Theorem 2]{RV} (see also \cite[Remark 3.12]{HS}) that there exist exactly $17$ level $2$ lattices of rank $16$ with determinant $2^6$ up to isometry and they are uniquely determined by their root system (see Table \ref{table2to6}).
Their isometry groups are determined by MAGMA as in Table \ref{table2to6}.
Hence there are $17$ possible lattices for $L=\sqrt2U^*$; indeed, they satisfy $(\mathcal{D}(L),q_L)\cong(\irr(W),-q_W)$.
Note that $O(L)= O(U)$.

\begin{longtable}[c]{|c|c|c|c|c|} 
\caption{Level $2$ lattices of rank $16$ for the case $2A$} \label{table2to6}\\
\hline 
$R(U_\g)$ &$U_\g/P_\g$& $O(U_\g)/W(R(U_\g))$& $O(U_\g)$\\ \hline 
\hline 
 $A_1^{16}$ &$\Z_2^5$& $\AGL_4(2)$& $W(A_1)\wr\AGL_4(2)$  \\ 
 $A_3^4({\sqrt2}A_1)^4$&$\Z_2^3\times\Z_4$ &$W(D_4)$ &$(W(A_3)^4\times W(A_1)^4).W(D_4)$\\ 
 $D_4^2C_2^4$  &$\Z_2^3$&$2\times\Sym_4$ &$(W(D_4)^2\times W(C_2)^4). (2\times\Sym_4)$\\ 

 $A_5^2({\sqrt2}A_2)^2 C_2$&$\Z_3\times\Z_6$ &$\Dih_8$ &$(W(A_5)^2\times W(A_2)^2\times W(C_2)).\Dih_8$ 
\\  
 $A_7 ({\sqrt2}A_3) C_3^2$ &$\Z_2\times\Z_4$&$\Z_2^2$ &$(W(A_7)\times W(A_3)\times W(C_3)^2).\Z_2^2$ 
\\
 $D_5^2 ({\sqrt2}A_3)^2$ &$\Z_4^2$&$\Dih_8$&$(W(D_5)^2\times W(A_3)^2).\Dih_8$\\ 

 $C_4^4$ &$\Z_2$&$\Sym_4$ &$W(C_4)\wr \Sym_4$\\ 
 $ D_6 C_4 B_{3}^2$ &$\Z_2^2$ & $\Z_2$&$(W(D_6)\times W(C_4)\times W(B_3)^2).\Z_2$\\  
 $A_9 ({\sqrt2A_4}){B_3}$ &$\Z_{10}$&$\Z_2$ &$(W(A_9)\times W(A_4)\times W(B_3)).\Z_2$
\\ 

 $E_6({\sqrt2A_5}) C_5$ &$\Z_6$
&$\Z_2$& $(W(E_6)\times W(A_5)\times W(C_5)).\Z_2$\\ 

 $D_8B_4^2$ &$\Z_2^2$&$\Z_2$&$W(D_8)\times W(B_4)\wr \Z_2$\\ 
 $C_6^2{B_4}$ &$\Z_2$&$\Z_2$ &$W(C_6)\wr 2\times W(B_4)$\\ 
 $D_9 ({\sqrt2A_7})$ &$\Z_8$&$\Z_2$&  $(W(D_9)\times W(A_7)).\Z_2$\\ 
 $C_8 F_4^2$ &$1$&$\Z_2$& $W(C_8)\times W(F_4)\wr \Z_2$\\
 $E_7{B_5}F_4$ &$\Z_2$&$1$& $W(E_7)\times W(B_5)\times W(F_4)$\\ 
 $C_{10}{B_6}$&$\Z_2$&$1$& $W(C_{10})\times W(B_6)$\\
 $E_8 {B_8}$&$\Z_2$ &$1$& $W(B_8)\times W(E_8)$\\ 
\hline 
\end{longtable}
Let $\g$ be one of the $17$ Lie algebras in Table \ref{T:Lie} corresponding to $2A$.
By Proposition \ref{P:root}, 
the root system $R(U)$ of $U=\sqrt{2}L^*$ is uniquely determined by $\g$ as in Table \ref{T:2to6}.
As we mentioned, $U$ is also uniquely determined by $\g$; we set $L_\g=L$ and $U_\g=U$.
Let $P_\g$ be the sublattice of $U_\g$ generated by $R(U_\g)$ as in \eqref{Eq:Ug} (see also Proposition \ref{P:root}).

\begin{proposition} Assume that the conjugacy class of $g$ is $2A$.
\begin{enumerate}[{\rm (1)}]
\item $K(V)\cong U_\g/P_\g$.
\item $\Out(V)\cong O(U_\g)/W(R(U_\g))$.
\end{enumerate}
\end{proposition}
\begin{proof}
By Proposition \ref{L:KV2}, we have $K(V)\cong L^*_\g/\tilde{Q}$.
It follows from the definition of $P_\g$ that $P_\g\cong\sqrt2\tilde{Q}$ (cf. Remark \ref{R:Pg}), which proves (1).
By Proposition \ref{P:root}, we obtain  $W(V_1)\cong W(R(U_\g))$.
Hence (2) follows from Lemma \ref{L:Out} (1).
\end{proof}

By the proposition above and Table \ref{table2to6}, we obtain the group structures of $K(V)$ and $\Out(V)$ for all $17$ cases, which are summarized in Table \ref{T:2to6}.

\begin{longtable}[c]{|c|c|c|c|c|c|c|}
\caption{$K(V)$ and $\Out(V)$ for the case $2A$} \label{T:2to6}\\
\hline 
Genus&No.&$\g=V_1$&$R(U_\g)$ &$  \OOut(V_1)$ &$\Out(V)$& $K(V)$ \\ \hline 
\hline 
$B$&$5$&$A_{1,2}^{16}$&$A_1^{16}$&$\Sym_{16}$ &$\AGL_4(2)$&  $\Z_2^5$  \\ 
&$16$&$A_{3,2}^4A_{1,1}^4$&$A_3^4(\sqrt2A_1)^4$&$(\Z_2\wr\Sym_4)\times\Sym_4$&$W(D_4)$& $\Z_2^3\times\Z_4$ \\  
&$25$&$D_{4,2}^2C_{2,1}^4$&$D_4^2C_2^4$&$(\Sym_3\wr\Sym_2)\times\Sym_4$  &$\Z_2\times\Sym_4$&$\Z_2^3$ \\ 

&$26$&$A_{5,2}^2C_{2,1}A_{2,1}^2$&$A_5^2C_2(\sqrt2A_2)^2$&$(\Z_2\wr\Sym_2)\times(\Z_2\wr\Sym_2)$ &$\Dih_8$  &$\Z_3\times\Z_6$
\\ 
&$31$&$D_{5,2}^2A_{3,1}^2$&$D_5^2(\sqrt2A_3)^2$&$(\Z_2\wr\Sym_2)\times(\Z_2\wr\Sym_2)$ &$\Dih_8$&$\Z_4^2$\\ 

&$33$&$A_{7,2}C_{3,1}^2A_{3,1}$&$A_7C_3^2(\sqrt2A_3)$&$\Z_2\times\Sym_2\times\Z_2$&$\Z_2^2$ &$\Z_2\times \Z_4$
\\ 
&$38$&$C_{4,1}^4$&$C_4^4$&$\Sym_4$ &$\Sym_4$&$\Z_2$ \\ 
&$39$&$D_{6,2}C_{4,1}B_{3,1}^2$&$D_6C_4B_3^2$&$\Z_2\times\Sym_2$  & $\Z_2$&$\Z_2^2$\\ 
&$40$&$A_{9,2}A_{4,1}B_{3,1}$&$A_9(\sqrt2A_4)B_3$&$\Z_2\times\Z_2$ &$\Z_2$ &$\Z_{10}$
\\ 
&$44$&$E_{6,2}C_{5,1}A_{5,1}$&$E_6C_5(\sqrt2A_5)$&$\Z_2\times\Z_2$
&$\Z_2$&$\Z_6$\\ 
&$47$&$D_{8,2}B_{4,1}^2$&$D_8B_4^2$&$\Z_2\times\Sym_2$ &$\Z_2$&$\Z_2^2$\\ 
&$48$&$C_{6,1}^2B_{4,1}$&$C_6^2B_4$&$\Sym_2$ &$\Z_2$&$\Z_2$ \\ 
&$50$&$D_{9,2}A_{7,1}$&$D_9(\sqrt2A_7)$&$\Z_2\times\Z_2$ &$\Z_2$&$\Z_8$\\ 
&$52$&$C_{8,1}F_{4,1}^2$&$C_8F_4^2$&$\Z_2$ &$\Z_2$&$1$\\
&$53$&$E_{7,2}B_{5,1}F_{4,1}$&$E_7B_5F_4$&$\Z_2$ &$1$&$\Z_2$\\ 
&$56$&$C_{10,1}B_{6,1}$&$C_{10}B_6$&$1$ &$1$&$\Z_2$\\
&$62$&$B_{8,1}E_{8,2}$&$B_8E_8$&$1$ &$1$&$\Z_2$\\ 
\hline 
\end{longtable}

\begin{remark} The groups $K(V)$ and $\Out(V)$ have been determined in \cite{Sh20} if $$\g\in\{A_{1,2}^{16},A_{3,2}^4A_{1,1}^4,D_{4,2}^2B_{2,1}^4,
D_{5,2}^2A_{3,1}^2,C_{4,1}^4,D_{6,2}B_{3,1}^2C_{4,1},
D_{8,2}B_{4,1}^2,D_{9,2}A_{7,1}\}$$ by using the explicit construction of $V$.
\end{remark}

\subsection{Conjugacy class $3B$ (Genus $C$)}
Assume that $g$ belongs to the conjugacy class $3B$ of $O(\Lambda)$.
Then $O(\irr(W),q_W)\cong GO_8^-(3)\cong 2\times \PO^-_8(3).2$.
By Table \ref{Table:main}, 
$\Aut(W)(\cong\overline{\Aut}(W))$ has the shape $\PO^-_8(3).2$, which is an index $2$ subgroup of $O(\irr(W),q_W)$.

Since the central charge of $W$ is $12$, $L$ is an even lattice of rank $12$ such that $(\mathcal{D}(L),q_L)\cong(\irr(W),-q_W)$.
Then $\mathcal{D}(L)\cong \Z_3^8$.
Set $U=\sqrt3L^*$.
Then $\mathcal{D}(U)\cong \Z_3^4$, and by Proposition \ref{P:root}, $U$ is a level $3$ lattice.
Such lattices $U$ were classified in \cite{RV}. 
Furthermore that can now be verified easily using MAGMA.
More precisely, it was proved in \cite[Theorem 3]{RV} that there exist exactly $6$ level $3$ lattices of rank $12$ with determinant $3^4$ up to isometry, and they are uniquely determined by their root system (see Table \ref{table3to4}).
Since $O(U)$ is a subgroup of the automorphism group of the root system $R(U)$, its shape is easily determined as in Table \ref{table3to4}.
Hence there are $6$ possible lattices for $L=\sqrt3U^*$; indeed, they satisfy $(\mathcal{D}(L),q_L)\cong(\irr(W),-q_W)$.
Note that $O(L)=O(U)$.

\begin{longtable}[c]{|c|c|c|c|} 
\caption{Level $3$ lattices of rank $12$ for the case $3B$} \label{table3to4}\\
\hline 
$R(U_\g)$ & $U_\g/P_\g$&$O(U_\g)/W(R(U_\g))$& $O(U_\g)$\\ \hline \hline 
$A_2^6$ & $\Z_3$&$\Z_2\times \Sym_6$ & $(W(A_2)\wr \Sym_6).\Z_2$\\ 
  $A_5D_4(\sqrt3A_1)^3$&$\Z_2^3$ &$\Dih_{12}$& $(W(A_5)\times W(D_4)\times W(A_1)^3).\Dih_{12}$
\\ 
$A_8 {(\sqrt3A_2)^2}$ &$\Z_3^2$&$\Z_2^2$& $(W(A_8)\times W(A_2)^2).\Z_2^2$ \\ 
$E_6 G_2^3$ &$1$&$\Z_2\times \Sym_3$&  $(W(E_6)\times W(G_2)\wr\Sym_3).\Z_2$ \\ 
$D_7 {(\sqrt3A_3)} G_2$&$\Z_4$&$\Z_2$&  $(W(D_7)\times W(A_3)\times W(G_2)).\Z_2$
\\
$E_7 {(\sqrt3A_5)}$ &$\Z_6$&$\Z_2$& $(W(E_7) \times W(A_5)).\Z_2$ \\ \hline
\end{longtable}
Let $\g$ be one of the $6$ Lie algebras in Table \ref{T:Lie} corresponding to $3B$. 
By Proposition \ref{P:root}, the root system of $U=\sqrt3L^*$ is uniquely determined by $\g$ as in Table \ref{table3to8}.
As we mentioned, $U$ is also uniquely determined by $\g$; we set $L_\g=L$ and $U_\g=U$.
Let $P_\g$ be the sublattice of $U_\g$ generated by $R(U_\g)$ as in \eqref{Eq:Ug} (see also Proposition \ref{P:root}).

\begin{proposition} Assume that the conjugacy class of $g$ is $3B$.
\begin{enumerate}[{\rm (1)}]
\item $K(V)\cong U_\g/P_\g$.
\item $\Out(V)\cong O(U_\g)/\langle W(R(U_\g)),-1\rangle$.
\end{enumerate}
\end{proposition}
\begin{proof}
By Proposition \ref{L:KV2}, we have $K(V)\cong L^*_\g/\tilde{Q}$.
It follows from the definition of $P_\g$ that $P_\g\cong\sqrt3\tilde{Q}$ (cf. Remark \ref{R:Pg}), which proves (1).
By Proposition \ref{P:root}, we obtain  $W(V_1)=W(R(U_\g))$.
Hence (2) follows from Proposition \ref{P:OV} and Lemma \ref{L:Out} (2).
\end{proof}

The group structures of $K(V)$ and $\Out(V)$ are summarized in Table \ref{table3to8}.
 
\begin{longtable}[c]{|c|c|c|c|c|c|c|} 
\caption{$K(V)$ and $\Out(V)$ for the case $3B$} \label{table3to8} \\
\hline 
Genus&No. &$\g=V_1$&$R(U_\g)$ &$  \OOut(V_1)$  &$\Out(V)$& $K(V)$\\ \hline 
\hline 
$C$&$6$&$A_{2,3}^{6}$&$A_2^6$ &$\Z_2\wr\Sym_6$    & $\Sym_6$&  $\Z_3$\\ 
&$17$&${A_{5,3}}{D_{4,3}}A_{1,1}^{3}$&$A_5D_4(\sqrt3A_1)^3$&$\Z_2\times\Sym_3\times\Sym_3$ & $\Sym_3$& $\Z_2^3$
\\ 
&$27$&$A_{8,3}A_{2,1}^2$&$A_8(\sqrt3A_2)^2$ & $\Z_2\times(\Z_2\wr\Sym_2)$    & $\Z_2$&$\Z_3^2$ \\ 
&$32$&$ E_{6,3}{G_{2,1}}^3$&$E_6 G_2^3$ & $\Z_2\times\Sym_3$  &  $\Sym_3$& $1$ \\
&$34$&$  {D_{7,3}}{A_{3,1}}{G_{2,1}}$&$D_7(\sqrt3A_3) G_2$ &  $\Z_2\times\Z_2$  
&  $1$& $\Z_4$
\\  
&$45$&$E_{7,3}A_{5,1}$ &$E_7 {(\sqrt3A_5)}$& $\Z_2$ & $1$ & $\Z_6$ \\ \hline
\end{longtable}

\subsection{Conjugacy class $5B$ (Genus $F$)}\label{order5}
Assume that $g$ belongs to the conjugacy class $5B$ of $O(\Lambda)$.
Then $O(\irr(W),q_W)\cong GO_6^+(5)\cong 2.\PO_6^+(5).2^2$.
By Table \ref{Table:main}, 
$\Aut(W)(\cong\overline{\Aut}(W))$ has the shape $2.\PO_6^+(5).2$, which is an index $2$ subgroup of $O(\irr(W),q_W)$ not isomorphic to $SO^+_6(5)$.

Since the central charge of $W$ is $16$, $L$ is an even lattice of rank $8$ such that $(\mathcal{D}(L),q_L)\cong(\irr(W),-q_W)$.
Then $\mathcal{D}(L)\cong\Z_5^6$.
Set $U=\sqrt5L^*$.
Then $\mathcal{D}(U)\cong\Z_5^2$, and by Proposition \ref{P:root}, $U$ is a level $5$ lattice.
Note that $O(L)=O(U)$.

By Table \ref{T:Lie}, the Lie algebra structure of $\g=V_1$ is $A_{4,5}^2$ or $D_{6,5}A_{1,1}^2$.
By Proposition \ref{P:root}, the root lattice $P_\g$ of $U$ is isometric to $A_4^2$ or $D_6(\sqrt5A_1^2)$, respectively.
It is easy to see that $U$ is uniquely determined as an overlattice of $P_\g$; we set $U_\g=U$ and $L_\g=L$.
Since $O(U)$ is a subgroup of the automorphism group of the root system $R(U)$, its shape is easily determined as in Table \ref{table5to2}.

\begin{longtable}[c]{|c|c|c|c|c|} 
\caption{Level $5$ lattices of rank $8$ for the case $5B$} \label{table5to2}\\
\hline 
$\g=V_1$&$R(U_\g)$ & $U_\g/P_\g$&$O(U_\g)/W(R(U_\g))$& $O(U_\g)$ \\ 
\hline
\hline
$A_{4,5}^2$&$A_4^2$ &$1$&$\Dih_8$&  $(2\times W(A_4))\wr\Sym_2$ \\
$D_{6,5}A_{1,1}^2$&$D_6(\sqrt5A_1^2)$ &$\Z_2^2$&$\Sym_2$& $(W(D_6) \times 
W(A_1)^2).2$ \\ \hline
\end{longtable}
 
\begin{lemma}\label{Lem:5b} The subgroups $\overline{O}(L_\g)$ and $\varphi^*(\overline{\Aut}(W))$ generate $O(\mathcal{D}(L_\g),q_{L_\g})$.
\end{lemma}
\begin{proof}
Let $\varphi$ be the isometry from $(\mathcal{D}(L_\g),q_{L_\g})$ to $(\irr(W),-q_W)$ satisfying \eqref{Eq:V}.
Recall that $\varphi^{*}(\overline{\Aut}(W))$ is an index $2$ subgroup of $O(\irr(W),q_W)\cong GO_6^+(5)$ not isomorphic to $SO^+_6(5)$ and that $O(\irr(W),q_W)\cong O(\irr(W),-q_W)$.
If the root system $R(U_\g)$ is $A_4^2$ or $D_6(\sqrt5A_1)^2$, then $\overline{O}(L_\g)\cong (2\times W(A_4))\wr\Sym_2\cong\GO_3(5)\wr\Sym_2$ or $\Z_2^6{:}\Sym_6\cong GO_1(5)\wr\Sym_6$, respectively.
In both cases, $\overline{O}(L_\g)$ is a maximal subgroup of $O(\mathcal{D}(L_\g),q_{L_\g})$ (cf.\ \cite[Theorem 3.12]{Wi}), and hence $\overline{O}(L_\g)$ and $\varphi^*(\overline{\Aut}(W))$ generate $O(\mathcal{D}(L_\g),q_{L_\g})$.
\end{proof}

By Proposition \ref{P:Hoext} and Lemma \ref{Lem:5b}, we obtain the following:

\begin{proposition}\label{P:uni5B} Assume that the conjugacy class of $g$ is $5B$.
Then, for each $L$, there exists exactly one holomorphic VOA of central charge $24$ obtained as a simple current extension of $V_L\otimes W$, up to isomorphism.
\end{proposition}

\begin{proposition} Assume that the conjugacy class of $g$ is $5B$.
\begin{enumerate}[{\rm (1)}]
\item $K(V)\cong U_\g/P_\g$.
\item $\Out(V)$ have the shapes in Table \ref{table5to6}.
\end{enumerate}
\end{proposition}
\begin{proof}
By Proposition \ref{L:KV2}, we have $K(V)\cong L^*_\g/\tilde{Q}$.
It follows from the definition of $P_\g$ that $P_\g\cong\sqrt5\tilde{Q}$ (cf. Remark \ref{R:Pg}), which proves (1).

Next, we determine $\Out(V)$.
By Proposition \ref{P:root}, we obtain  $W(V_1)=W(R(U_\g))$.
By Proposition \ref{P:OV} and Lemmas \ref{L:Out} and \ref{Lem:5b}, we have $|O(L_\g)/W(R(U_\g)):\Out(V)|=2$.
Hence $|\Out(V)|=4$ or $1$ if $R(U_\g)\cong A_4^2$ or $D_6(\sqrt5A_1)^2$, respectively.

Assume that $R(U_\g)\cong A_4^2$.
Note that $\overline{O}(L_\g)\cong O(L_\g)$
and that $ \overline{O}(L_\g)\cap \varphi^{*}(\overline{\Aut}(W))$ contains $W(R(U_\g))\cong \Sym_5\times\Sym_5$ as a subgroup.
Checking possible index $2$ subgroups of $\overline{O}(L_\g)$ obtained as $\overline{O}(L_\g)\cap\varphi^*(\overline{\Aut}(W))$, one can verify that $\overline{O}(L_\g)\cap\varphi^{*}(\overline{\Aut}(W))$ has the shape $2\times(\Sym_5\wr\Sym_2)$ by using MAGMA.
Hence $\Out(V)\cong\Z_2^2$ by Proposition \ref{P:OV}.
\end{proof}

The group structures of $K(V)$ and $\Out(V)$ are summarized in Table \ref{table5to6}.

\begin{longtable}[c]{|c|c|c|c|c|c|c|} 
\caption{$K(V)$ and $\Out(V)$ for the case $5B$} \label{table5to6}\\
\hline 
Genus&No.&$\g=V_1$ &$R(U_\g)$ &$  \OOut(V_1)$  &$\Out(V)$& $K(V)$\\ 
\hline \hline 
$F$&$9$&$ A_{4,5}^2$&$A_4^2$ &$\Z_2\wr\Sym_2$   & $\Z_2^2$ & $1$ \\
&$20$&$D_{6,5}A_{1,1}^2$&$D_6(\sqrt5A_1^2)$ &$\Z_2\times\Sym_2$   & $1$& $\Z_2^2$ \\ \hline
\end{longtable}

\begin{remark}\label{R:5B} The even lattice $U$ is a level $5$ lattice of rank $8$ (see Proposition \ref{P:root}) and $(\mathcal{D}(U),q_U)$ is a $2$-dimensional quadratic space over $\Z_5$ of plus type.
By using these properties, one could prove that the root system of $U$ is $A_4^2$ or $D_6(\sqrt{5}A_1^2)$.
\end{remark}

\subsection{Conjugacy class $7B$ (Genus $H$)}\label{order7}
Assume that $g$ belongs to the conjugacy class $7B$ of $O(\Lambda)$.
Then $O(\irr(W),q_W)\cong GO_5(7)\cong 2\times \PO_5(7).2$.
By Table \ref{Table:main},
$\Aut(W)(\cong\overline{\Aut}(W))$ has the shape $\PO_5(7).2$, which is an index $2$ subgroup of $O(\irr(W),q_W)$. 

By Table \ref{T:Lie}, the Lie algebra structure of $\g=V_1$ is $A_{6,7}$. 
Since the central charge of $W$ is $18$, 
$L$ is an even lattice of rank $6$ such that $(\mathcal{D}(L),q_L)\cong(\irr(W),-q_W)$.
Then $\mathcal{D}(L)\cong \Z_7^5$.
Set $U=\sqrt7L^*$.
Then $\mathcal{D}(U)\cong \Z_7$ and by Proposition \ref{P:root}, $U$ is a level $7$ lattice. 
In addition, the root system $R(U)$ of $U$ is $A_6$.
Hence $U=U_\g\cong P_\g\cong A_6$, and $L=L_\g\cong\sqrt{7}A_6^*$.
The isometry group of $U_\g$ is summarized in Table \ref{table7}.

\begin{longtable}[c]{|c|c|c|c|c|} 
\caption{Level $7$ lattice of rank $6$ for the case $7B$} \label{table7}\\
\hline 
$\g=V_1$&$R(U_\g)$ &$U_\g/P_\g$& $O(U_\g)/W(R(U_\g))$& $O(U_\g)$ \\ 
\hline
\hline
$A_{6,7}$&$A_6$ &$1$&$\Z_2$&  $\Z_2\times W(A_6)$ \\ \hline
\end{longtable}

\begin{proposition} Assume that the conjugacy class of $g$ is $7B$.
\begin{enumerate}[{\rm (1)}]
\item $K(V)\cong U_\g/P_\g$.
\item $\Out(V)=1$
\end{enumerate}
\end{proposition}
\begin{proof}

By Proposition \ref{L:KV2}, we have $K(V)\cong L^*_\g/\tilde{Q}$.
It follows from the definition of $P_\g$ that $P_\g\cong \sqrt7\tilde{Q}$ (cf. Remark \ref{R:Pg}), which proves (1).
By Proposition \ref{P:root}, we have $W(V_1)=W(R(U_\g))$.
Hence (2) follows from Proposition \ref{P:OV} and Lemma \ref{L:Out} (2).
\end{proof}

The group structures of $K(V)$ and $\Out(V)$ are summarized in Table \ref{table7to5}.

\begin{longtable}[c]{|c|c|c|c|c|c|c|} 
\caption{$K(V)$ and $\Out(V)$ for the case $7B$} \label{table7to5}\\
\hline 
Genus&No.&$\g=V_1$&$R(U_\g)$ &$  \OOut(V_1)$  &$\Out(V)$& $K(V)$\\ 
\hline \hline 
$H$&$11$&$ A_{6,7}$&$A_6$ & $\Z_2$ & $1$ & $1$  \\ \hline
\end{longtable}

\begin{remark}\label{R:7B} The lattice $U$ is an even lattice of rank $6$ with $\mathcal{D}(U)\cong\Z_7$.
By using this property, one could prove that $U\cong A_6$.
\end{remark}

\subsection{Conjugacy class $2C$ (Genus $D$)}\label{Sec:2c} 
Assume that $g$ belongs to the conjugacy class $2C$ of $O(\Lambda)$.
Then $O(\irr(W),q_W)\cong 2^{1+20}_+.(GO^-_{10}(2)\times\Sym_3)$.
By Table \ref{Table:main}, $\Aut(W)(\cong\overline{\Aut}(W))$ has the shape $ 2^{1+20}_+.(\Sym_{12}\times \Sym_3)$.
Since $\Sym_{12}$ is a maximal subgroup of $GO^-_{10}(2)$ (\cite{ATLAS}), $\Aut(W)$ is also a maximal subgroup of 
$O(\irr(W),q_W)$ and it is self-normalizing.

\begin{remark}\label{R:S12}
Set $\Omega=\{1,2,\dots,12\}$ and $\mathcal{X}=\{A\subset\Omega\mid |A|\equiv0\pmod2\}/\{\emptyset,\Omega\}$.
Then $\mathcal{X}$ is a $10$-dimensional vector space over $\F_2$ by the symmetric difference.
In addition, $\mathcal{X}$ has the quadratic form of minus type defined by $A\mapsto |A|/2\pmod2$ (cf.\ \cite[page 147]{ATLAS}).
Since $\Sym_{12}$ naturally acts on $\mathcal{X}$ and it preserves the quadratic form, we obtain $\Sym_{12}\subset GO^-_{10}(2)$.
\end{remark}

It was proved in \cite[Theorem 2.8]{HS} that $(\irr(W),q_W)\cong(\mathcal{D}(\sqrt2D_{12}),q_{\sqrt2D_{12}})$.
Since the central charge of $W$ is $12$, the rank of $L$ is $12$.
Note that $(\irr(W),-q_W)\cong (\irr(W),q_W)$ and that $(\mathcal{D}(L),q_L)\cong(\irr(W),-q_W)$.

\begin{lemma}\label{L:2C}
Let $H$ be an even lattice of rank $12$ such that $(\mathcal{D}(H), q_H) \cong (\mathcal{D}(\sqrt{2}D_{12} ), q_{\sqrt2D_{12}})$ as quadratic spaces. Then $H\cong \sqrt{2}D_{12}$ or $\sqrt{2}E_8\sqrt2D_4$.  
\end{lemma}

\begin{proof} It follows from $\mathcal{D}(H)\cong\Z_2^{10}\times\Z_4^2$ and $\rank H=12$ that $(1/2)H\subset H^*$.
In addition, it also follows from $q_{\sqrt2D_{12}}((1/2)(\sqrt2{D}_{12}))\subset \Z$ that $q_H((1/2)H)\subset\Z$.
Hence $(1/\sqrt2)H$ is even.
Since $\mathcal{D}((1/\sqrt2)H)\cong\Z_2^2$ and the rank of $(1/\sqrt2)H$ is $12$, there exists an odd unimodular lattice of rank $12$ whose even sublattice is $(1/\sqrt2)H$.
Since any odd unimodular lattice of rank $12$ is isometric to $\Z^{12}$ or $E_8\Z^4$ (cf. \cite{CS}), we have $(1/\sqrt2)H\cong D_{12}$ or $E_8D_4$
\end{proof}

By this lemma, we have $L\cong\sqrt2D_{12}$ or $L\cong\sqrt2E_8\sqrt2D_4$.
We will discuss each case in the following subsections.

\subsubsection{Case $L\cong \sqrt{2}D_{12}$}\label{S451}
In this case, 
$\mu_L$ is injective, that is, $\overline{O}(L)\cong O(L)\cong 2^{12}.\Sym_{12}$. 
Recall that $O(\mathcal{D}(L),q_L)\cong O(\irr(W),q_W)\cong 2^{1+20}_+.(GO^-_{10}(2)\times\Sym_3)$.
Note that $\overline{O}(L)\cap O_2(O(\mathcal{D}(L),q_L))\cong 2^{11}$, where $O_2(O(\mathcal{D}(L),q_L))\cong 2^{1+20}_+$ is the maximal normal $2$-subgroup of $O(\mathcal{D}(L),q_L)$.
Fix an isometry $i: (\mathcal{D}(L),q_L) \to  (\irr(W), -q_W)$.
By Lemma \ref{DCvsConj} and Remark \ref{DCcount}, the number of double cosets of  $$ i^*( \overline{\Aut}(W))\backslash O(\mathcal{D}(L),q_L) /\overline{O}(L)$$
 is equal to the number of $2^{12}.\Sym_{12}$-orbits on the set of all $O(\irr(W),q_W)$-conjugates of the subgroup $2^{1+20}_+.(\Sym_{12}\times\Sym_3)$. Since $2^{1+20}_+.\Sym_3$ is a normal subgroup of $O(\irr(W),q_W)$ and the quotient of $O(\irr(W),q_W)$ by this normal subgroup is $GO_{10}^-(2)$, 
this number is also equal to the number of $\Sym_{12}$-orbits on its conjugates in $GO_{10}^-(2)$. 
By Remark \ref{R:S12}, there is a natural embedding $\Sym_{12}\subset GO^-_{10}(2)$, and by \cite[page 147]{ATLAS} (cf.\ \cite[Remark 2.9]{HS}), there exist six $\Sym_{12}$-orbits on its conjugates in $GO_{10}^-(2)$.
Hence we obtain the following lemma.

\begin{lemma}\label{2C1}
There exist exactly $6$ double cosets in $ i^*( \overline{\Aut}(W))\backslash O(\mathcal{D}(L),q_L) /\overline{O}(L)$. 
\end{lemma}

By Proposition \ref{P:Hoext} and the lemma above, we obtain $6$ holomorphic VOAs of central charge $24$ as inequivalent simple current extensions.
In fact, their weight one Lie algebras are non-isomorphic, which is discussed in \cite[Remark 2.12]{HS} and \cite[Table 8]{Ho2} (see Table \ref{2c1} for the Lie algebra structures).
Hence we obtain the following:

\begin{proposition}\label{P:uni2Ca}
Assume that the conjugacy class of $g$ is $2C$ and that $L\cong\sqrt{2}D_{12}$.
Then there exist exactly $6$ holomorphic VOAs of central charge $24$ obtained as inequivalent simple current extensions of $V_L\otimes W$, up to isomorphism.
\end{proposition}

In \cite[Page 147]{ATLAS}, the shapes of the $6$ subgroups of $GO^-_{10}(2)$ obtained as the intersection of two maximal subgroups isomorphic to $\Sym_{12}$ are described.
These groups appear as the quotient of
$\overline{O}(L)\cap \varphi^{*}(\overline{\Aut}(W))$ by $O_2(\overline{O}(L))\cong 2^{12}$ for isometries $\varphi$ from $(\mathcal{D}(L),q_L)$ to $(\irr(W), -q_W)$.
By Proposition \ref{P:OV}, we obtain $\Out(V)$ as in Table \ref{2c1}.
For any weight one Lie algebra structure in Table \ref{2c1}, we have $\tilde{Q}\cong (1/\sqrt2)\Z^{12}$.
By Proposition \ref{L:KV2} and $L^*\cong (1/\sqrt2)D_{12}^*$, we have $K(V)\cong\Z_2$.

\begin{longtable}[c]{|c|c|c|c|c|c|c|c|} 
\caption{$K(V)$ and $\Out(V)$ for the case $2C$ and $L\cong \sqrt{2}D_{12}$} \label{2c1}\\
\hline 
Genus&No.& $V_1$ & $\overline{O}(L)\cap \varphi^{*}(\overline{\Aut}(W))$ & $W(V_1)$ &$  \OOut(V_1)$& $\Out(V)$&$K(V)$ \\ 
\hline 
\hline 
$D$&$2$ &$A_{1,4}^{12}$ & $2^{12}.M_{12}$  & $W(A_1)^{12} $  & $\Sym_{12}$& $M_{12}$&   $\Z_2$\\
&$12$ &$B_{2,2}^6$ & $2^{12}.(\Z_2^6: \Sym_5)$ & $W(B_2)^6$ & $\Sym_6$& $\Sym_5$& $\Z_2$\\
&$23$ &$B^4_{3,2}$ & $2^{12}.(\Sym_{3}\wr \Alt_4)$ & $ W(B_3)^4 $ &$\Sym_4$& $\Alt_4$&   $\Z_2$\\
&$29$ &$B_{4,2}^3$ & $2^{12}.(\Sym_{4}\wr \Sym_3)$ & $ W(B_4)^3 $ & $\Sym_3$&  $\Sym_3$& $\Z_2$ \\
&$41$ &$B^2_{6,2}$ & $2^{12}.(\Sym_{6}\wr 2)$& $W(B_6)^2$ & $\Sym_2$& $\Z_2$& $\Z_2$  \\
&$57$ &$B_{12,2}$ & $2^{12}.\Sym_{12}$& $W(B_{12})$ & $1$& $1$ &  $\Z_2$  \\ 
\hline 
\end{longtable}

\begin{remark} For the cases in Table \ref{2c1}, the groups $K(V)$ and $\Out(V)$ have been determined in \cite{Sh20} by using the explicit construction of $V$.
\end{remark}

\subsubsection{Case $L\cong \sqrt{2}E_8\sqrt2D_4$}
In this case, $O(L)\cong O(D_4)\times O(E_8)\cong ((2^{1+4}.\Sym_3){:}\Sym_3)\times 2.GO^+_8(2)$.
In addition, $O_0(L)$ is generated by the $-1$-isometry of $\sqrt2E_8$ and $\overline{O}(L)\cong O(D_4)\times (O(E_8)/\langle -1\rangle)\cong ((2^{1+4}.\Sym_3){:}\Sym_3)\times GO^+_8(2)$. 
We can rewrite as $\overline{O}(L)\cong2^{1+4}.((\Sym_3\times GO^+_8(2))\times\Sym_3)$, which corresponds to the shape of $O(\irr(W),q_W)
\cong 2^{1+20}_+.(GO^-_{10}(2)\times \Sym_3)$.

Fix an isometry $i: (\mathcal{D}(L),q_L) \to  (\irr(W), -q_W)$.
By Lemma \ref{DCvsConj}, the number of double cosets is equal to the number of $2^{1+4}.((\Sym_3\times GO^+_8(2))\times\Sym_3)$-orbits on the set of conjugates of $2^{1+20}_+.(\Sym_{12}\times\Sym_3)$.
It follows from $O_2(O(\irr(W),q_W))\cong 2^{1+20}_+$ that the number is also equal to the number of $(\Sym_3\times GO^+_8(2))$-orbits on the set of conjugates of the subgroup $\Sym_{12}\subset GO_{10}^{-}(2)$. 

\begin{lemma}\label{2C2}
There exist exactly $3$ double cosets in $ i^*( \overline{\Aut}(W))\backslash O(\mathcal{D}(L),q_L) /\overline{O}(L)$.
\end{lemma}
\begin{proof} 
Recall from Remark \ref{R:S12} the construction of a $10$-dimensional quadratic space $\mathcal{X}$ over $\F_2$ of minus type with natural embedding $\Sym_{12}\subset GO_{10}^-(2)$.
It is well-known (cf. \cite{ATLAS,Wi}) that the stabilizer in $GO_{10}^-(2)$ of a non-singular $2$-space of minus-type is a maximal subgroup of the shape $\Sym_3\times GO^+_8(2)$.
By the definition of the quadratic form in Remark \ref{R:S12}, non-singular vectors of $\mathcal{X}$ are $2$-sets or $6$-sets modulo $\{\emptyset,\Omega\}$.
Then there exist exactly three orbits $\mathcal{Q}_i$ ($i=1,2,3$) of non-singular $2$-spaces of minus type in $\mathcal{X}$ under the action of $\Sym_{12}$.
Here the non-zero vectors of $\mathcal{Q}_1$, $\mathcal{Q}_2$, $\mathcal{Q}_3$ are three $2$-sets, three $6$-sets, or one $2$-set and two $6$-sets, respectively.
One can then deduce that there exist exactly $3$ $(\Sym_3\times GO^+_8(2))$-orbits on the set of conjugates of the subgroup $\Sym_{12}\subset GO_{10}^-(2)$. 
\end{proof}

By Proposition \ref{P:Hoext} and this lemma, we obtain $3$ holomorphic VOAs of central charge $24$ as inequivalent simple current extensions.
In fact, their weight one Lie algebras are non-isomorphic, which is discussed in \cite[Remark 2.12]{HS} and \cite[Table 8]{Ho2} (see Table \ref{2c1b} for the Lie algebra structures.)
Hence we obtain the following:

\begin{proposition}\label{P:uni2Cb}
Assume that the conjugacy class of $g$ is $2C$ and that $L\cong\sqrt{2}E_8\sqrt2D_4$.
Then there exist exactly $3$ holomorphic VOAs of central charge $24$ obtained as inequivalent simple current extensions of $V_L\otimes W$, up to isomorphism.
\end{proposition}

{\tiny
\begin{longtable}[c]{|c|c|c|c|c|c|c|c|} 
\caption{$K(V)$ and $\Out(V)$ for case $2C$ and $L\cong \sqrt{2}E_8\sqrt2D_4$} \label{2c1b}\\
\hline 
Genus&No. & $V_1$& $\mu_{L}^{-1}(\overline{O}(L)\cap \varphi^{*}(\overline{\Aut}(W)))$ & $W(V_1)$ &$  \OOut(V_1)$& $\Out(V)$ &$K(V)$\\ 
\hline 
\hline 
$D$&$13$ &$D_{4,4}A^4_{2,2}$ & $2.((W(D_4)\times W(A_2)^4).\Sym_4)$ & $W(D_4)\times W(A_2)^4$&$\Sym_3\times\Z_2\wr\Sym_4$ & $2.\Sym_4$&  $\Z_3^2$\\
&$22$ &$C_{4,2}A^2_{4,2}$& $W(C_4)\times 2.(W(A_4)^2).2$& $W(C_4)\times W(A_4)^2$ &$\Z_2\wr\Sym_2$& $\Z_4$& $\Z_5$ \\
&$36$ &$ A_{8,2}F_{4,2}$ & $2.(W(A_8)\times W(F_4))$& $W(A_8)\times W(F_4)$&$\Z_2$  & $\Z_2$  &   $\Z_3$ \\ 

\hline 
\end{longtable}
}

Let $\mathcal{Q}_i$ $(i=1,2,3)$ be non-singular $2$-spaces of minus type in $\mathcal{X}$ given in the proof of Lemma \ref{2C2}.
Then the stabilizers of $\mathcal{Q}_1$, $\mathcal{Q}_2$ and $\mathcal{Q}_3$ in $\Sym_{12}$ are $\Sym_3\times\Sym_9$, $\Sym_3\wr\Sym_4$ and $\Sym_2\times\Sym_5\wr\Sym_2$, respectively.
Let $\varphi_i$ be an isometry from $(\mathcal{D}(L),q_L)$ to $(\irr(W), -q_W)$ associated with $\mathcal{Q}_i$. 
Then $\overline{O}(L)\cap \varphi^*_i(\overline{\Aut}(W))$ has the shapes $2^{1+4}.(\Sym_3\times\Sym_9).\Sym_3$, $2^{1+4}.(\Sym_3\wr\Sym_4).\Sym_3$ and $2^{1+4}.(\Sym_2\times\Sym_5\wr\Sym_2).\Sym_3$, respectively.
By the shapes of these groups, the corresponding Lie algebra structures of $V_1$ are $A_{8,2}F_{4,2}$, $D_{4,4}A_{2,2}^4$ and $C_{4,2}A_{4,2}^2$, respectively.
Note that the Weyl groups of the simple ideals of type $F_{4,2}$, $D_{4,4}$ and $C_{4,2}$ act as the diagram automorphism group $\Sym_3$ on $\sqrt{2}D_4\subset L$, respectively.
The group $\mu_L^{-1}(\overline{O}(L)\cap\varphi^*_i(\overline{\Aut}(W)))$ is a central extension of $\overline{O}(L)\cap\varphi^*_i(\overline{\Aut}(W))$ by $O_0(L)\cong\Z_2$.
Since $\mu_L^{-1}(\overline{O}(L)\cap\varphi^*_i(\overline{\Aut}(W)))$ contains $W(V_1)$, we can rewrite the shapes of $\mu_L^{-1}(\overline{O}(L)\cap\varphi^*_i(\overline{\Aut}(W)))$ as $2.(W(A_8)\times W(F_4))$, $2.(W(D_4)\times W(A_2)^4).\Sym_4$ and $2.(W(C_4)\times (W(A_4)^2).2)$, respectively.
Note that for the case $C_{4,2}A_{4,2}^2$, the subgroup $2.(W(A_4)^2).2$ is the stabilizer in $O(E_8)$ of the sublattice $A_4^2$ of $E_8$, and its quotient by $W(A_4)^2$ is isomorphic to $\Z_4$.
Hence we obtain the shape of $\Out(V)$ as in  Table \ref{2c1b} by Proposition \ref{P:OV}.

Recall that $L^*\cong (1/\sqrt2)D_4^*(1/\sqrt2)E_8$.
If the Lie algebra structure of $V_1$ is $A_{8,2}F_{4,2}$, $C_{4,2}A_{4,2}^2$ or $D_{4,4}A_{2,2}^4$, then $\tilde{Q}$ is isometric to $(1/\sqrt2)A_8(1/\sqrt2)D_4^*$, $(1/2)D_4(1/\sqrt2)A_4^2$ and $(1/2)D_4(1/\sqrt2)A_2^4$, respectively.
Note that $(1/\sqrt2)D_4^*\cong(1/2)D_4$.
By Proposition \ref{L:KV2}, we obtain $K(V)\cong L^*/\tilde{Q}$ as in Table \ref{2c1b}.

\begin{proposition} Assume that $g$ belongs to the conjugacy class $2C$.
The shapes of the groups $K(V)$ and $\Out(V)$ are given as in Tables \ref{2c1} and \ref{2c1b}.
\end{proposition}

\subsection{Conjugacy class $4C$ (Genus $E$)}
Assume that $g$ belongs to the conjugacy class $4C$ of $O(\Lambda)$.
Then $O(\irr(W),q_W)\cong 2^{22}.GO_7(2)$.
By Table \ref{Table:main}, $\Aut(W))(\cong\overline{\Aut}(W))$ has the shape $2^{21}.GO_7(2)$, which is an index $2$ subgroup of $O(\irr(W),q_W)$.
Note also that the Lie algebra structure of $\g=V_1$ is given as in Table \ref{T:Lie}.

Since the central charge of $W$ is $14$, the rank of $L$ is $10$.  
By \eqref{Eq:LQl} and Proposition \ref{P:Ug}, we have $Q_\g\subset L\subset \sqrt4P_\g^*$.
It follows from Table \ref{Table:main} and $\mathcal{D}(L)\cong\irr(W)$ that $\mathcal{D}(L)\cong \Z_2^2\times\Z_4^6$.
For each $\g$, its Lie algebra structure gives the lattices $Q_g$ and $\sqrt4P_\g^*$ as in Table \ref{4Ctable}.
Then one can easily see that there exists a unique even lattice $L$ up to isometry satisfying $\mathcal{D}(L)\cong \Z_2^2\times\Z_4^6$ and $Q_\g\subset L\subset \sqrt4P_\g^*$;
see Table \ref{4Ctable} for the explicit description.
We set $L_\g=L$.
The isometry groups of $L_\g$ are also summarized in  Table \ref{OK4C}.

\begin{remark}\label{R:Glue}
Let us explain the meaning of ``Glue'' in the tables.
Let $Q_\g=\bigoplus_{i=1}^s c_i R_i$, where $c_i\in\R$ and $R_i$ are irreducible root lattices.
In our cases, $L_\g$ is a sublattice of $\bigoplus_{i=1}^sc_i R_i^*$; we associate $L_\g/Q_\g$ to a subgroup of $\bigoplus_{i=1}^s (R_i^*/R_i)$ via the inclusion $L_\g/Q_\g\subset (\bigoplus_{i=1}^s c_i R_i^*)/(\bigoplus_{i=1}^s c_i R_i)\cong \bigoplus_{i=1}^s (R_i^*/R_i)$. 
In the tables, based on the isomorphisms
$A^*_m/A_m\cong \Z_{m+1}$, $D_{2m+1}^*/D_{2m+1}\cong\Z_4$, $D_{2m}^*/D_{2m}\cong\Z_2^2=\langle b,c\rangle$ and $E_6^*/E_6\cong\Z_3$, the generators of $L_\g/Q_\g$ is described as a subgroup of $\bigoplus_{i=1}^s (R_i^*/R_i)$ in ``Glue''.
Here $b,c$ are chosen so that they are permuted by the diagram automorphism of order $2$.
\end{remark}

By Proposition \ref{P:OV} and Lemma \ref{L:Out} (2), we have $\Out(V)\cong O(L_\g)/\langle W(V_1),-1\rangle$.
The group $K(V)$ is determined by Proposition \ref{L:KV2}.
These structures are summarized in Table \ref{OK4C}.

\begin{proposition} Assume that $g$ belongs to the conjugacy class $4C$.
Then the shapes of the groups $K(V)$ and $\Out(V)$ are given as in Table \ref{OK4C}.
\end{proposition}

{\tiny
\begin{longtable}[c]{|c|c|c|c|c|c|} 
\caption{Even lattices of rank $10$ for the case $4C$.} \label{4Ctable}\\
\hline 
$\mathfrak{g}=V_1$&$Q_\g$&$\sqrt4P_\g^*$ &$L_\g/Q_\g$& Glue &$O(L_\g)$\\
\hline \hline
$A_{3,4}^3A_{1,2}$ & $(2A_3)^3\sqrt{2}A_1$ &$(2A_3^*)^3\sqrt2A_1^*$& $\Z_4^3$& $(100;1),(010;1)$& $O(A_3)\wr\Sym_3\times W(A_1)$\\ 
&&&&$(001;1)$&\\

$A_{7,4}A_{1,1}^3$ & $2A_7 A_1^3$ &$2A_7^*(A_1^*)^3$&$\Z_8$& $(1;100)$&  $O(A_7)\times W(A_1)\times W(A_1)\wr\Sym_2$\\  
$D_{5,4}C_{3,2}A_{1,1}^2$ & $2D_5 \sqrt{2}A_1^3A_1^2$ &$2D_5^*2A_3^*(A_1^*)^2$& $\Z_4\times\Z_2$ & $(1;111;00)$&  $O(D_5)\times O(A_3)\times W(A_1)\wr\Sym_2$ 
\\
&&&&$(0;111;11)$&\\

$E_{6,4}A_{2,1}B_{2,1}$  & $2E_6 A_2 A_{1}^2 $&$2E_6^*A_2^*A_1^2$ &$\Z_3$& $(1;1;00)$&  $(W(E_6)\times W(A_1)).2\times W(B_2)$ \\
$C_{7,2}A_{3,1}$ &$\sqrt2A_1^7A_3$&$2D_7^*A_{3}^*$ & $\Z_2$&$(1^7;2)$ &  $O(D_7)\times O(A_3)$  \\
\hline
\end{longtable}
}

{\small
\begin{longtable}[c]{|c|c|c|c|c|c|} 
\caption{$K(V)$ and $\Out(V)$ for the case $4C$} \label{OK4C}\\
\hline 
No.& $\g=V_1$ &  $W(V_1)$&$ \OOut(V_1)$& $\Out(V)$ &$K(V)$ \\ 
\hline \hline
$7$ &$A_{3,4}^3A_{1,2}$ & $ W(A_3)^3\times W(A_1)$& $\Z_2\wr\Sym_3$ &$\Z_2^2:\Sym_3$& $\Z_2$
 \\

$18$ &$A_{7,4}A_{1,1}^3$  & $W(A_7)\times W(A_1)^3$& $\Z_2\times\Sym_3$ & $\Z_2$  & $\Z_2^3$\\  
$19$ &$D_{5,4}C_{3,2}A_{1,1}^2$  &$ W(D_5)\times W(C_3)\times W(A_1)^2$&$ \Z_2\times\Sym_2$ &$\Z_2$& $\Z_2^3$\\
$28$ &$E_{6,4}A_{2,1}B_{2,1}$ &  $W(E_6)\times W(A_2)\times W(B_2)$& $\Z_2\times\Z_2$ &$1$& $\Z_6$\\
$35$ &$C_{7,2}A_{3,1}$  & $W(C_7)\times W(A_3)$&$\Z_2$ & $1$ & $\Z_2^2$
\\ \hline
\end{longtable}
}

\subsection{Conjugacy class $6E$ (Genus $G$)}
Assume that $g$ belongs to the conjugacy class $6E$ of $O(\Lambda)$.
Then $O(\irr(W),q_W)\cong GO_6^+(2)\times GO_6^+(3)$.
By Table \ref{Table:main}, $\Aut(W)(\cong\overline{\Aut}(W))$ is isomorphic to the full orthogonal group $O(\irr(W),q_W)$.
Note also that the Lie algebra structure of $\g=V_1$ is given as in Table \ref{T:Lie}.

Since the central charge of $W$ is $16$, the rank of $L$ is $8$.
By \eqref{Eq:LQl} and Proposition \ref{P:Ug}, we have $Q_\g\subset L\subset \sqrt6P_\g^*$.
It follows from Table \ref{Table:main} and $\mathcal{D}(L)\cong\irr(W)$ that $\mathcal{D}(L)\cong \Z_2^6\times\Z_3^6$.
For each $\g$, its Lie algebra structure gives the lattices $Q_\g$ and $\sqrt6P_\g^*$ as in Table \ref{table6to6}.
Then one can easily see that there exists a unique even lattice $L$ up to isometry satisfying $\mathcal{D}(L)\cong \Z_2^6\times\Z_3^6$ and $Q_\g\subset L\subset \sqrt6P_\g^*$;
see Table \ref{table6to6} for the explicit description.
Set $L_\g=L$.
The isometry groups of $L_\g$ are also summarized in  Table \ref{table6to6}.

The group $\Out(V)$ is determined by Lemma \ref{L:Out} (1),
 and the group $K(V)$ is determined by Proposition \ref{L:KV2}.
These structures are summarized in Table \ref{OK6E}.

\begin{proposition} Assume that $g$ belongs to the conjugacy class $6E$.
Then the shapes of the groups $K(V)$ and $\Out(V)$ are given as in Tables \ref{OK6E}.
\end{proposition}

{\tiny
\begin{longtable}[c]{|c|c|c|c|c|c|} 
\caption{Even lattices of rank $8$ for the case $6E$} \label{table6to6}\\
\hline 
$\mathfrak{g}=V_1$  & $Q_\g$ & $\sqrt{6}P_\g^*$ & $L_\g/Q_\g$ &Glue &$O(L_\g)$\\ 
\hline \hline
$A_{5,6}B_{2,3}A_{1,2}$  & $\sqrt{6}A_5 \sqrt{3} A_{1}^2 \sqrt{2}A_1$ &$\sqrt{6}A_5^*(\sqrt3A_1^*)^2\sqrt2A_1^*$& $\Z_6\times\Z_2$ & $(1;00;1)$ & $O(A_5)\times O(A_1^2)\times W(A_1)$\\
&&&&$(0;11;1)$&\\ 
$C_{5,3}G_{2,2}A_{1,1}$ & $\sqrt{3}A_1^5 \sqrt{2}A_2 A_1$ &$\sqrt6D_5^*\sqrt2A_2^*A_1^*$&  $\Z_2$&$(1^5;0;1)$  & $O(A_5)\times O(A_1^2)\times O(A_1)$
\\ 
\hline 
\end{longtable}
}

{\small
\begin{longtable}[c]{|c|c|c|c|c|c|} 
\caption{$K(V)$ and $\Out(V)$ for the case $6E$} \label{OK6E}\\
\hline 
No.& $\g=V_1$& $W(V_1)$&$  \OOut(V_1)$& $\Out(V)$ &$K(V)$ \\ 
\hline \hline
$8$ & $A_{5,6}B_{2,3}A_{1,2}$  & $W(A_5)\times W(B_2)\times W(A_1)$ 
&$\Z_2$
&$\Z_2$& $\Z_2$\\
$21$ & $C_{5,3}G_{2,2}A_{1,1}$ &  $W(C_5)\times W(G_2)\times W(A_1)$&
$1$& $1$ & $\Z_2$\\ \hline
\end{longtable}
}

\subsection{Conjugacy class $6G$ (Genus $J$)}
Assume that $g$ belongs to the conjugacy class $6G$ of $O(\Lambda)$.
Then $O(\irr(W),q_W)\cong 2^{1+8}_+{:}(GO_4^+(2)\times\Sym_3)\times GO_5(3)$.
Here $GO_4^+(2)\cong\Sym_3\wr\Z_2$ and $GO_5(3)\cong 2\times \PO_5(3).2$.
By Table \ref{Table:main}, $\Aut(W)(\cong\overline{\Aut}(W))$ has the shape $2^{1+8}_+{:}(\Sym_3\times\Sym_3\times\Sym_3)\times \PO_5(3).2$, which is 
an index $4$ subgroup of $O(\irr(W),q_W)$.

\begin{remark} Let $O_2(\irr(W))$ be the Sylow $2$-subgroup of $\irr(W)$ of shape $2^4.4^2$.
Then $O(O_2(\irr(W)),q_W)\cap \Aut(W)\cong 2^{1+8}_+{:}(3\times\Sym_3\times\Sym_3)$, which is computed by MAGMA.
Hence we can rewrite $\Aut(W)\cong (2^{1+8}_+{:}(3\times\Sym_3\times\Sym_3)\times \PO_5(3).2).2$ with respect to the shape of $O(\irr(W),q_W)$.
In fact, $\Aut(W)$ is not normal in $O(\irr(W),q_W)$ by MAGMA.
\end{remark}

By Table \ref{T:Lie}, the Lie algebra structure of $\g=V_1$ is $F_{4,6}A_{2,2}$ or $D_{4,12}A_{2,6}$ .
By \eqref{Eq:LQl} and Proposition \ref{P:Ug}, we have $Q_\g\subset L\subset \sqrt{12}P_\g^*$.
It follows from Table \ref{Table:main} and $\mathcal{D}(L)\cong\irr(W)$ that $\mathcal{D}(L)\cong \Z_2^4\times\Z_4^2\times\Z_3^5$.
In both cases, we have $L\cong \sqrt{6}D_4\sqrt{2}A_2\cong\sqrt{12}D_4^*\sqrt6A_2^*$ and $O(L)\cong O(D_4)\times O(A_2)$
(see Table \ref{tablelevel12}).
Note that $\overline{O}(L)\cong O(L)$.

Let $i:(\mathcal{D}(L),q_L)\to(\irr(W),-q_W)$ be an isometry.
By the possible Lie algebra structures of $\g$, there exist at least two non-isomorphic holomorphic VOAs obtained as inequivalent extensions of $V_L\otimes W$.
Since $\overline{\Aut}(W)$ is an index $4$ subgroup of $O(\irr(W),q_W)$, we have $$2\le |i^*(\overline{\Aut}(W))\setminus O(\mathcal{D}(L),q_L)/\overline{O}(L)|\le4$$ by Proposition \ref{P:Hoext}.
By Lemma \ref{Lem:center} (2), the $-1$-isometry is not in $i^*(\overline{\Aut}(W))$.
Clearly it is in $\overline{O}(L)$.
Hence, 
$$|i^*(\overline{\Aut}(W))\setminus O(\mathcal{D}(L),q_L)/\overline{O}(L)|=2.$$
By Proposition \ref{P:Hoext}, we obtain the following:
\begin{proposition}\label{P:uni6G} Assume that the conjugacy class of $g$ is $6G$.
Then there exist exactly two holomorphic VOAs of central charge $24$ obtained as inequivalent simple current extensions of $V_L\otimes W$, up to isomorphism.
\end{proposition}

By the argument above, $i^*(\overline{\Aut}(W))$ is an index $2$ subgroup of the group generated by $i^*(\overline{\Aut}(W))$ and $\overline{O}(L)$.
Thus $\overline{O}(L)\cap i^*(\overline{\Aut}(W))$ is an index $2$ subgroup of $\overline{O}(L)$, and by Lemma \ref{Lem:center} (2), $\overline{O}(L)\cap i^*(\overline{\Aut}(W))\cong \overline{O}(L)/\langle-1\rangle$.
By Proposition \ref{P:OV} and $\overline{O}(L)\cong O(L)$, we have
$\Out(V)\cong O(L)/\langle W(V_1),-1\rangle$.
The group $K(V)$ is determined by Proposition \ref{L:KV2}.
These group structures are summarized in Table \ref{OK6G}.

\begin{proposition} Assume that $g$ belongs to the conjugacy class $6G$.
Then 
the shapes of the groups $K(V)$ and $\Out(V)$ are given as in Table \ref{OK6G}.
\end{proposition}

{\small
\begin{longtable}[c]{|c|c|c|c|c|c|} 
\caption{Even lattice of rank $6$ for the case $6G$} \label{tablelevel12}\\
\hline 
$\g=V_1$  & $Q_\g$&$\sqrt{12}P_\g^*$   & $L/Q_\g$&Glue&$O(L)$\\ 
\hline \hline
$D_{4,12}A_{2,6}$ & $\sqrt{12}D_4 \sqrt{6}A_2$&$\sqrt{12}D_4^*\sqrt6A_2^*$&$\Z_2^2\times\Z_3$ & $(b;0),(c;0),(0;1)$&$O(D_4)\times O(A_2)$\\
$F_{4,6}A_{2,2}$  & $\sqrt{6}D_4 \sqrt{2}A_2$ &$\sqrt{12}D_4^*\sqrt2A_2^*$ & $1$&$1$&$O(D_4)\times O(A_2)$ \\

 \hline 
\end{longtable}
}

{\small \begin{longtable}[c]{|c|c|c|c|c|c|} 
\caption{$K(V)$ and $\Out(V)$ for the case $6G$} \label{OK6G}\\
\hline 
No.& $\g=V_1$ & $W(V_1)$&$  \OOut(V_1)$& $\Out(V)$ &$K(V)$ \\ 
\hline \hline
$3$ &$D_{4,12}A_{2,6}$   &$W(D_4)\times W(A_2)$ &$\Sym_3\times\Z_2$& $\Sym_3$ & $1$\\
$14$ &$F_{4,6}A_{2,2}$   &$W(F_4)\times W(A_2)$&$\Z_2$ & $1$ & $\Z_3$  \\

 \hline 
\end{longtable}
}

\subsection{Conjugacy class $8E$ (Genus $I$)}
Assume that $g$ belongs to the conjugacy class $8E$ of $O(\Lambda)$.
Then $O(\irr(W),q_W)\cong 2^{12+9}.\Sym_6$.
By Table \ref{Table:main}, $\Aut(W)(\cong\overline{\Aut}(W))$ has the shape $2^{11+9}.\Sym_6$, which is
an index $2$ subgroup of $O(\irr(W),q_W)$.
Note also that the Lie algebra structure of $\g=V_1$ is $D_{5,8}A_{1,2}$ by Table \ref{T:Lie}.

Since the central charge of $W$ is $18$, the rank of $L$ is $6$. 
By \eqref{Eq:LQl} and Proposition \ref{P:Ug}, we have $Q_\g\subset L\subset \sqrt8P_\g^*$.
It follows from Table \ref{Table:main} and $\mathcal{D}(L)\cong\irr(W)$ that $\mathcal{D}(L)\cong \Z_2\times\Z_4\times\Z_8^4$.
Hence, we have $L_\g=L\cong\sqrt{8}D_5^*\sqrt2A_1$ and $O(L_\g)\cong O(D_5)\times W(A_1)$ (see Table \ref{tablelevel8}).

By Proposition \ref{P:OV} and Lemma \ref{L:Out} (2), we have $\Out(V)\cong O(L_\g)/\langle W(V_1),-1\rangle$.
The group $K(V)$ is determined by Proposition \ref{L:KV2}. 
See Table \ref{OK8E} for the structures.

\begin{proposition} Assume that $g$ belongs to the conjugacy class $8E$.
Then 
the shapes of the groups $K(V)$ and $\Out(V)$ are given as in Table \ref{OK8E}.
\end{proposition}

{\small
\begin{longtable}[c]{|c|c|c|c|c|c|} 
\caption{Even lattice of rank $6$ for the case $8E$} \label{tablelevel8}\\
\hline 
$V_1=\mathfrak{g}$  & $Q_\g$ &$\sqrt8P_\g^*$& $L_\g/Q_\g$ &Glue&$O(L_\g)$ \\ 
\hline \hline
$D_{5,8}A_{1,2}$  & $\sqrt{8}D_5 \sqrt{2}A_1$&$\sqrt8D_5^*\sqrt2A_1^*$&$\Z_4$& $(1;0)$& $O(D_5)\times W(A_1)$  \\
 \hline 
\end{longtable}
}

{\small
\begin{longtable}[c]{|c|c|c|c|c|c|} 
\caption{$K(V)$ and $\Out(V)$ for the case $8E$} \label{OK8E}\\
\hline 
No.& $\g=V_1$ & $W(V_1)$& $  \OOut(V_1)$&$\Out(V)$ &$K(V)$ \\ 
\hline \hline
$10$ & $D_{5,8}A_{1,2}$   &  $W(A_5)\times W(A_1)$ &$\Z_2$& $1$ &  $\Z_2$\\
 \hline 
\end{longtable}
}

\subsection{Conjugacy class $10F$ (Genus $K$)}
Assume that $g$ belongs to the conjugacy class $10F$ of $O(\Lambda)$.
Then $O(\irr(W),q_W)\cong 2^{1+4}_+{:}(\Sym_3\times\Sym_3)\times GO_4^+(5)$.
By Table \ref{Table:main}, 
$\Aut(W)(\cong\overline{\Aut}(W))$ has the shape $2^{1+4}_+{:}(2\times\Sym_3)\times GO_4^+(5)$, which  
is an index $3$ subgroup of $O(\irr(W),q_W)$.
By Table \ref{T:Lie}, the Lie algebra structure of $\g=V_1$ is $C_{4,10}$.

Since the central charge of $W$ is $20$, the rank of $L$ is $4$.
By \eqref{Eq:LQl} and Proposition \ref{P:Ug}, we have $Q_\g\subset L\subset \sqrt{20}P_\g^*$.
It follows from Table \ref{Table:main} and $\mathcal{D}(L)\cong\irr(W)$ that $\mathcal{D}(L)\cong \Z_2^2\times\Z_4^2\times\Z_5^4$.
Then we have $L_\g=L\cong\sqrt{10}D_4$ and $O(L_\g)\cong O(D_4)$
(see Table \ref{tablelevel20}).

\begin{remark}\label{R:10F}
For $U=\sqrt{20} L^*$, we have $\mathcal{D}(U)= \Z_2^2$ and ${\rm rank}(U)=4$. It is easy to show that $U \cong D_4$ and thus $L=\sqrt{20}U^* \cong \sqrt{20}D_4^*\cong \sqrt{10} D_4$. 
\end{remark}

Since $V_1\cong C_{4,10}$, we have $  \OOut(V_1)=1$, and $\Out(V)=1$.
The group $K(V)$ is trivial by Proposition \ref{L:KV2}.
These group structures are summarized in Table \ref{tablelevel10}.

\begin{proposition} Assume that $g$ belongs to the conjugacy class $10F$.
Then 
the shapes of the groups $K(V)$ and $\Out(V)$ are given as in Table \ref{tablelevel10}.
\end{proposition}

{\small
	\begin{longtable}[c]{|c|c|c|c|c|c|} 
		\caption{Even lattice of rank $4$ for the case $10F$} \label{tablelevel20}\\
		\hline 
		$\g=V_1$  & $Q_\g$&$\sqrt{20}P_\g^*$ & $L_\g/Q_\g$&Glue&$O(L_\g)$\\ 
		\hline \hline
		$C_{4,10}$  & $\sqrt{10}A_1^4$  &$\sqrt{20}D_4^*$ & $\Z_2$&$(1111)$& $O(D_4)$ \\
		\hline 
	\end{longtable}
}

{\small
	\begin{longtable}[c]{|c|c|c|c|c|c|} 
		\caption{$K(V)$ and $\Out(V)$ for the case $10F$} \label{tablelevel10}\\
		\hline 
		No.& $V_1$ & $W(V_1)$ &$  \OOut(V_1)$ &  $\Out(V)$ &$K(V)$\\ 
		\hline \hline
		$4$ &$C_{4,10}$  & $W(C_4)$ & $1$ & $1$ &$1$\\
		\hline 
	\end{longtable}
}

It is easy to see that $\mu_L$ is injective, that is, $\overline{O}(L)\cong O(L)$.
Let $\varphi$ be an isometry from $(\mathcal{D}(L),q_L)$ to $(\irr(W),-q_W)$.
By Proposition \ref{P:OV} and $\Out(V)=1$, we have $\overline{O}(L)\cap\varphi^*(\overline{\Aut}(W))=W(V_1)$.
Since $W(V_1)(\cong W(C_4))$ is an index $3$ subgroup of $\overline{O}(L)$ (see Lemma \ref{Lem:W}), so is $\overline{O}(L)\cap\varphi^*(\overline{\Aut}(W))$.
Hence $\varphi^*(\overline{\Aut}(W))$ and $\overline{O}(L)$ generate $O(\mathcal{D}(L),q_L)$.
By Proposition \ref{P:Hoext}, we obtain the following:
\begin{proposition}\label{P:uni10F} Assume that the conjugacy class of $g$ is $10F$.
Then, there exists exactly one holomorphic VOA of central charge $24$ obtained as inequivalent simple current extensions of $V_L\otimes W$, up to isomorphism.
\end{proposition}

As a consequence of our calculations, we have proved Theorem \ref{T:main13} and confirmed \cite[Conjecture 4.8]{Ho2}. 
Combining with the characterization of Niemeier lattice VOAs in \cite{DM04b}, it provides another proof for the uniqueness of holomorphic vertex operator algebras of central charge $24$ with non-trivial weight one Lie algebras.

\appendix\section{Actions of automorphism groups on the weight one spaces}

In this appendix, for holomorphic VOAs $V$ of central charge $24$ whose weight one Lie algebras are semisimple, we describe the subgroup $\Out_1(V)$ of $\Out(V)$ which preserves every simple ideal of $V_1$ and the quotient group $\Out_2(V)=\Out(V)/\Out_1(V)$. 

\subsection{Simple current modules over $L_{\hat{\g}}(k,0)$}

Let $\g$ be a simple Lie algebra and let $k$ be a positive integer.
Let $L_{\hat{\g}}(k,0)$ be the simple affine VOA associated with $\g$ at level $k$.
Let $S_\g$ be the set of isomorphism classes of simple current $L_{\hat{\g}}(k,0)$-modules.
Then $S_\g$ has an abelian group structure under the fusion product.
The structures of $S_\g$ are well-known (see  \cite[Remark 2.21]{Li01} and reference therein), which are summarized in Table \ref{T:scaffine}.
Here $\Gamma(\g)$ is the diagram automorphism group of $\g$ and $[\Lambda]$ is the irreducible $L_{\hat{\g}}(k,0)$-module $L_{\hat{\g}}(k,\Lambda)$.
Note that the notations $[i](=i[1])$, $[s]$ and $[c]$ are used in \cite{Sc93}.

\begin{longtable}[c]{|c|c|c|c|c|}
\caption{Simple current $L_{\hat{\g}}(k,0)$-modules} \label{T:scaffine}\\
\hline 
Type & level& $S_\g$&$\Gamma(\g)$& generators of $S_\g$\\ \hline 
\hline 
$A_1$ &$k$& $\Z_{2}$&$1$& $[1]=[k\Lambda_1]$\\
$A_n$ $(n\ge2)$&$k$& $\Z_{n+1}$&$\Z_2$& $[1]=[k\Lambda_1]$\\
$B_n$ ($n\ge2$)&$k$&$\Z_2$&$1$&$[1]=[k\Lambda_1]$\\
$C_n$ ($n\ge2)$&$k$&$\Z_2$&$1$&$[1]=[k\Lambda_n]$\\
$D_{4}$ &$k$&$\Z_2\times\Z_2$&$\Sym_3$&$[s]=[k\Lambda_{n-1}],[c]=[k\Lambda_{n}]$\\
$D_{2n}$ $(n\ge3)$&$k$&$\Z_2\times\Z_2$&$\Z_2$&$[s]=[k\Lambda_{n-1}],[c]=[k\Lambda_{n}]$\\
$D_{2n+1}$ $(n\ge2)$&$k$&$\Z_4$&$\Z_2$&$[s]=[k\Lambda_{n-1}]$\\
$E_6$&$k$&$\Z_3$&$\Z_2$&$[1]=[k\Lambda_1]$\\
$E_7$&$k$&$\Z_2$&$1$&$[1]=[k\Lambda_6]$\\
$E_8$&$2$&$\Z_2$&$1$&$[1]=[\Lambda_7]$\\
$E_8$&$k\neq2$&$1$&$1$&\\
$F_4$&$k$&$1$&$1$&\\
$G_2$&$k$&$1$&$1$&\\
\hline
\end{longtable}

\subsection{Glue codes of holomorphic VOAs of central charge $24$}

Let $V$ be a holomorphic VOA of central charge $24$ with $0<\rank V_1<24$.
Let $\Out_1(V)$ be the subgroup of $\Out(V)$ which preserves every simple ideal of $V_1$ and set $\Out_2(V)=\Out(V)/\Out_1(V)$.
Then $\Out_2(V)$ is the permutation group on the set of simple ideals of $V_1$ induced from $\Out(V)$.

Let $V_1=\bigoplus_{i=1}^s\g_i$ be the direct sum of simple ideals.
Let $S_i(=S_{{\g}_i})$ be the set of (the isomorphism classes of) simple current $L_{\hat{\g}_i}(k_i,0)$-modules, where $k_i$ is the level of $\g_i$ in $V$.
Then $S_i$ has an abelian group structure under the fusion product.

Let $S_\g=\prod_{i=1}^s S_i$ be the direct product of the groups $S_i$.
We often view a simple current $\langle V_1\rangle$-module as an element of $S_\g$ via the map $\bigotimes_{i=1}^sM^i\mapsto (M^1,\dots,M^s)$.
Let $\{1,2,\dots,s\}=\bigcup_{b\in B} I_b$ be the partition such that $\g_i\cong\g_j$ if and only if $i,j\in I_b$ for some $b\in B$, where $B$ is an index set.
The automorphism group $\Aut(S_\g)$ of $S_\g$ is defined to be $(\prod_{i=1}^s\Gamma(\g_i)):(\prod_{b\in B} \Sym_{|I_b|})$, where the symmetric group $\Sym_{|I_b|}$ acts naturally on $\prod_{i\in I_b}S_i$.

Let $G_V$ be the subgroup of $S_\g$ consisting of all (isomorphism classes of) simple current $\langle V_1\rangle$-submodules of $V$, which we call the \emph{Glue code} of $V$.
The automorphism group $\Aut(G_V)$ of $G_V$ is defined to be the subgroup of $\Aut(S_\g)$ stabilizing $G_V$.
Let $\Aut_1(G_V)=(\prod_{i=1}^s\Gamma(\g_i))\cap\Aut(G_V)$ and $\Aut_2(G_V)=\Aut(G_V)/\Aut_1(G_V)$.
Then $\Aut_1(G_V)$ is the subgroup of $\Aut(G_V)$ stabilizing every $S_i$, and $\Aut_2(G_V)$ acts faithfully on $\{S_i\mid 1\le i\le s\}$, or $\{\g_i\mid 1\le i\le s\}$, as a permutation group.
Clearly $\Aut(V)$ preserves $S_\g$.
Hence $\Out_i(V)\subset\Aut_i(G_V)$ for $i=1,2$.

By using the generators of the glue codes $G_V$ in \cite{Sc93}, we can easily determine $\Aut_1(G_V)$ and $\Aut_2(G_V)$ explicitly.
We also determine the shapes of $\Out_1(V)$ and $\Out_2(V)$; see Tables \ref{T:2A}.

\begin{longtable}[c]{|c|c|c|c|c|c|c|}
\caption{$\Aut_i(G_V)$ and $\Out_i(V)$} \label{T:2A}\\
\hline 
No.&Genus&$\g=V_1$&$\Aut_1(G_V)$&$\Out_1(V)$ &$\Aut_2(G_V)$ &$\Out_2(V)$ \\ \hline 
\hline 
$15$&$A$&$A_{1,1}^{24}$&$1$&$1$&$M_{24}$&$M_{24}$\\
$24$&&$A_{2,1}^{12}$&$\Z_2$&$\Z_2$&$M_{12}$&$M_{12}$\\
$30$&&$A_{3,1}^8$&$\Z_2$&$\Z_2$&$AGL_3(2)$&$AGL_3(2)$\\
$37$&&$A_{4,1}^6$&$\Z_2$&$\Z_2$&$\Sym_5$&$\Sym_5$\\
$42$&&$D_{4,1}^6$&$\Z_3$&$\Z_3$&$\Sym_6$&$\Sym_6$\\
$43$&&$A_{5,1}^4D_{4,1}$&$\Z_2$&$\Z_2$&$\Sym_4$&$\Sym_4$\\
$46$&&$A_{6,1}^4$&$\Z_2$&$\Z_2$&$\Alt_4$&$\Alt_4$\\
$49$&&$A_{7,1}^2D_{5,1}^2$&$\Z_2$&$\Z_2$&$\Z_2^2$&$\Z_2^2$\\
$51$&&$A_{8,1}^3$&$\Z_2$&$\Z_2$&$\Sym_3$&$\Sym_3$\\
$54$&&$D_{6,1}^4$&$1$&$1$&$\Sym_4$&$\Sym_4$\\
$55$&&$A_{9,1}^2D_{6,1}$&$\Z_2$&$\Z_2$&$\Z_2$&$\Z_2$\\
$58$&&$E_{6,1}^4$&$\Z_2$&$\Z_2$&$\Sym_4$&$\Sym_4$\\
$59$&&$A_{11,1}D_{7,1}E_{6,1}$&$\Z_2$&$\Z_2$&$1$&$1$\\
$60$&&$A_{12,1}^2$&$\Z_2$&$\Z_2$&$\Z_2$&$\Z_2$\\
$61$&&$D_{8,1}^3$&$1$&$1$&$\Sym_3$&$\Sym_3$\\
$63$&&$A_{15,1}A_{9,1}$&$\Z_2$&$\Z_2$&$1$&$1$\\
$64$&&$D_{10,1}E_{7,1}^2$&$1$&$1$&$\Z_2$&$\Z_2$\\
$65$&&$A_{17,1}E_{7,1}$&$1$&$1$&$1$&$1$\\
$66$&&$D_{12,1}^2$&$1$&$1$&$\Z_2$&$\Z_2$\\
$67$&&$A_{24,1}$&$\Z_2$&$\Z_2$&$1$&$1$\\
$68$&&$E_{8,1}^3$&$1$&$1$&$\Sym_3$&$\Sym_3$\\
$69$&&$D_{16,1}E_{8,1}$&$1$&$1$&$1$&$1$\\
$70$&&$D_{24,1}$&$1$&$1$&$1$&$1$\\
\hline
$5$&$B$&$A_{1,2}^{16}$&$1$&$1$&$\AGL_4(2)$&  $\AGL_4(2)$  \\ 
$16$&&$A_{3,2}^4A_{1,1}^4$&$\Z_2$&$\Z_2$&$\Z_2^4:\Sym_3$& $\Z_2^4:\Sym_3$ \\  
$25$&&$D_{4,2}^2C_{2,1}^4$&$1$&$1$  &$\Sym_2\times\Sym_4$&$\Sym_2\times\Sym_4$ \\ 

$26$&&$A_{5,2}^2C_{2,1}A_{2,1}^2$&$\Z_2$&$\Z_2$ &$\Sym_2\times\Sym_2$  &$\Sym_2\times\Sym_2$
\\ 
$31$&&$D_{5,2}^2A_{3,1}^2$&$\Z_2$&$\Z_2$ &$\Sym_2\times\Sym_2$&$\Sym_2\times\Sym_2$\\ 

$33$&&$A_{7,2}C_{3,1}^2A_{3,1}$&$\Z_2$ &$\Z_2$&$\Sym_2$ &$\Sym_2$
\\ 
$38$&&$C_{4,1}^4$&$1$&$1$ &$\Sym_4$&$\Sym_4$ \\ 
$39$&&$D_{6,2}C_{4,1}B_{3,1}^2$&$1$&$1$  & $\Sym_2$&$\Sym_2$\\ 
$40$&&$A_{9,2}A_{4,1}B_{3,1}$&$\Z_2$&$\Z_2$ &$1$ &$1$
\\ 

$44$&&$E_{6,2}C_{5,1}A_{5,1}$&$\Z_2$&$\Z_2$
&$1$&$1$\\ 

$47$&&$D_{8,2}B_{4,1}^2$&$1$&$1$ &$\Sym_2$&$\Sym_2$\\   
$48$&&$C_{6,1}^2B_{4,1}$&$1$&$1$ &$\Sym_2$&$\Sym_2$ \\ 
$50$&&$D_{9,2}A_{7,1}$&$\Z_2$&$\Z_2$ &$1$&$1$\\ 
$52$&&$C_{8,1}F_{4,1}^2$&$1$&$1$ &$\Sym_2$&$\Sym_2$\\
$53$&&$E_{7,2}B_{5,1}F_{4,1}$&$1$&$1$ &$1$&$1$\\ 
$56$&&$C_{10,1}B_{6,1}$&$1$&$1$ &$1$&$1$\\
$62$&&$B_{8,1}E_{8,2}$&$1$&$1$ &$1$&$1$\\ 
\hline 
$6$&$C$&$A_{2,3}^{6}$&$\Z_2$ &$1$&$\Sym_6$&$\Sym_6$\\ 
$17$&&${A_{5,3}}{D_{4,3}}A_{1,1}^{3}$&$\Z_2$&$1$&$\Sym_3$ & $\Sym_3$
\\ 
$27$&&$A_{8,3}A_{2,1}^2$&$\Z_2$ & $1$    & $\Sym_2$&$\Sym_2$ \\ 
$32$&&$ E_{6,3}{G_{2,1}}^3$&$\Z_2$&$1$&$\Sym_3$  &  $\Sym_3$ \\
$34$&&$  {D_{7,3}}{A_{3,1}}{G_{2,1}}$&$\Z_2$ &  $1$  
&  $1$& $1$
\\  
$45$&&$E_{7,3}A_{5,1}$ &$\Z_2$&$1$& $1$ & $1$  \\ \hline
$2$&$D$ &$A_{1,4}^{12}$ & $1$  & $1$  & $\Sym_{12}$& $M_{12}$\\
$12$ &&$B_{2,2}^6$ & $1$ & $1$ & $\Sym_6$& $\Sym_5$\\
$23$ &&$B^4_{3,2}$ & $1$ & $1$ &$\Sym_4$& $\Alt_4$\\
$29$ &&$B_{4,2}^3$ & $1$ & $1 $ & $\Sym_3$&  $\Sym_3$ \\
$41$ &&$B^2_{6,2}$ & $1$& $1$ & $\Sym_2$& $\Sym_2$  \\
$57$ &&$B_{12,2}$ & $1$& $1$ & $1$& $1$   \\ 
$13$ &&$D_{4,4}A^4_{2,2}$ & $\Sym_3\times\Z_2$ & $\Z_2$&$\Sym_4$ & $\Sym_4$\\
$22$ &&$C_{4,2}A^2_{4,2}$& $\Z_2$& $\Z_2$ &$\Sym_2$& $\Sym_2$ \\
$36$ &&$ A_{8,2}F_{4,2}$ & $\Z_2$& $\Z_2$&$1$  & $1$ \\ 
\hline
$7$ &$E$&$A_{3,4}^3A_{1,2}$ & $\Z_2^3$& $\Z_2^2$ &$\Sym_3$& $\Sym_3$
 \\
$18$ &&$A_{7,4}A_{1,1}^3$  & $\Z_2$& $1$ & $\Z_2$  & $\Z_2$\\  
$19$ &&$D_{5,4}C_{3,2}A_{1,1}^2$  &$\Z_2$&$ 1$ &$\Z_2$& $\Z_2$\\
$28$ &&$E_{6,4}A_{2,1}B_{2,1}$ &  $\Z_2$& $1$ &$1$& $1$\\
$35$ &&$C_{7,2}A_{3,1}$  & $\Z_2$&$1$ & $1$ & $1$\\
\hline
$9$&$F$&$A_{4,5}^2$&$ \Z_2^2$&$\Z_2$ &$\Sym_2$   & $\Sym_2$ \\
$20$&&$D_{6,5}A_{1,1}^2$&$1$ &$1$&$\Sym_2$   & $1$ \\ \hline
$8$&$G$ & $A_{5,6}B_{2,3}A_{1,2}$  & $\Z_2$ 
&$\Z_2$
&$1$& $1$\\
$21$ && $C_{5,3}G_{2,2}A_{1,1}$ &  $1$&
$1$& $1$ & $1$\\ \hline
$11$&$H$&$ A_{6,7}$&$\Z_2$ & $1$ & $1$ & $1$  \\ \hline
$10$ &$I$& $D_{5,8}A_{1,2}$   &  $\Z_2$ &$1$& $1$ &  $1$\\
\hline
$3$ &$J$&$D_{4,12}A_{2,6}$   &$\Sym_3\times\Z_2$ &$\Sym_3$& $1$ & $1$\\
$14$ &&$F_{4,6}A_{2,2}$   &$\Z_2$&$1$ & $1$ & $1$  \\
\hline
$4$&$K$ &$C_{4,10}$  & $1$ & $1$ & $1$ &$1$\\
\hline 
\end{longtable}

\end{document}